\documentclass[12pt]{amsart}
\usepackage{}
\usepackage{amsmath}
\usepackage{amsfonts}
\usepackage{amssymb}
\usepackage[all,cmtip]{xy}           
\usepackage{bbm}
\usepackage{bbding}
\usepackage{txfonts}
\usepackage[shortlabels]{enumitem}
\usepackage{ifpdf}
\ifpdf
\usepackage[colorlinks,final,backref=page,hyperindex]{hyperref}
\else
\usepackage[colorlinks,final,backref=page,hyperindex,hypertex]{hyperref}
\fi
\usepackage{amscd}
\usepackage{tikz-cd}
\usepackage[active]{srcltx}

\makeatletter

\newtheorem{thm}{Theorem}[section]
\newtheorem{prop}[thm]{Proposition}
\newtheorem{lem}[thm]{Lemma}

\newtheorem{prop-def}{Proposition-Definition}[section]
\theoremstyle{definition}
\newtheorem{defn}[thm]{Definition}

\newtheorem{remark}[thm]{Remark}
\newtheorem{exam}[thm]{Example}
\newtheorem{Ques}[thm]{Question}

\newcommand{\nc}{\newcommand}

\nc{\delete}[1]{{}}
\nc{\mmargin}[1]{}
\nc{\mlabel}[1]{\label{#1}}  
\nc{\mcite}[1]{\cite{#1}}  
\nc{\mref}[1]{\ref{#1}}  
\nc{\mbibitem}[1]{\bibitem{#1}} 

\delete{
	\nc{\mlabel}[1]{\label{#1}  
		{\hfill \hspace{1cm}{\bf{{\ }\hfill(#1)}}}}
	\nc{\mcite}[1]{\cite{#1}{{\bf{{\ }(#1)}}}}  
	\nc{\mref}[1]{\ref{#1}{{\bf{{\ }(#1)}}}}  
	\nc{\mbibitem}[1]{\bibitem[\bf #1]{#1}} 
}

 \font\cyrs=wncyr7

\newcommand{\bk}{{\mathbf{k}}}

\nc{\vep}{\varepsilon}
\nc{\bin}[2]{ (_{\stackrel{\scs{#1}}{\scs{#2}}})}  
\nc{\binc}[2]{(\!\! \begin{array}{c} \scs{#1}\\
		\scs{#2} \end{array}\!\!)}  
\nc{\bincc}[2]{  ( {\scs{#1} \atop
		\vspace{-1cm}\scs{#2}} )}  
\nc{\oline}[1]{\overline{#1}}
\nc{\mapm}[1]{\lfloor\!|{#1}|\!\rfloor}
\nc{\bs}{\bar{S}}
\nc{\la}{\longrightarrow}
\nc{\ot}{\otimes}
\nc{\rar}{\rightarrow}
\nc{\lon }{\,\rightarrow\,}
\nc{\dar}{\downarrow}
\nc{\dap}[1]{\downarrow \rlap{$\scriptstyle{#1}$}}
\nc{\defeq}{\stackrel{\rm def}{=}}
\nc{\dis}[1]{\displaystyle{#1}}
\nc{\dotcup}{\ \displaystyle{\bigcup^\bullet}\ }
\nc{\hcm}{\ \hat{,}\ }
\nc{\hts}{\hat{\otimes}}
\nc{\hcirc}{\hat{\circ}}
\nc{\lleft}{[}
\nc{\lright}{]}
\nc{\curlyl}{\left \{ \begin{array}{c} {} \\ {} \end{array}
	\right .  \!\!\!\!\!\!\!}
\nc{\curlyr}{ \!\!\!\!\!\!\!
	\left . \begin{array}{c} {} \\ {} \end{array}
	\right \} }
\nc{\longmid}{\left | \begin{array}{c} {} \\ {} \end{array}
	\right . \!\!\!\!\!\!\!}
\nc{\ora}[1]{\stackrel{#1}{\rar}}
\nc{\ola}[1]{\stackrel{#1}{\la}}
\nc{\scs}[1]{\scriptstyle{#1}} \nc{\mrm}[1]{{\rm #1}}
\nc{\dirlim}{\displaystyle{\lim_{\longrightarrow}}\,}
\nc{\invlim}{\displaystyle{\lim_{\longleftarrow}}\,}
\nc{\dislim}[1]{\displaystyle{\lim_{#1}}} \nc{\colim}{\mrm{colim}}
\nc{\mvp}{\vspace{0.3cm}} \nc{\tk}{^{(k)}} \nc{\tp}{^\prime}
\nc{\ttp}{^{\prime\prime}} \nc{\svp}{\vspace{2cm}}
\nc{\vp}{\vspace{8cm}}
\nc{\modg}[1]{\!<\!\!{#1}\!\!>}
\nc{\intg}[1]{F_C(#1)}
\nc{\lmodg}{\!<\!\!}
\nc{\rmodg}{\!\!>\!}
\nc{\cpi}{\widehat{\Pi}}
\nc{\ssha}{{\mbox{\cyrs X}}} 
\nc{\tsha}{{\mbox{\cyrt X}}}
\nc{\shpr}{\diamond}    
\nc{\labs}{\mid\!}
\nc{\rabs}{\!\mid}
 \nc{\zhx}{\text{-}}

\nc{\ad}{\mrm{ad}}
\nc{\ann}{\mrm{ann}}
\nc{\Aut}{\mrm{Aut}}
\nc{\Av}{\mrm{Av}}
\nc{\bim}{\mbox{-}\mathsf{Bimod}}
\nc{\br}{\mrm{bre}}
\nc{\can}{\mrm{can}}
\nc{\Cont}{\mrm{Cont}}
\nc{\rchar}{\mrm{char}}
\nc{\cok}{\mrm{coker}}
\nc{\db}{\mrm{db}}
\nc{\de}{\mrm{dep}}
\nc{\dgg}{\mrm{dgg}}
\nc{\dgp}{\mrm{dgp}}
\nc{\dgx}{\mrm{dgx}}
\nc{\Dif}{\mrm{Diff}}
\nc{\dtf}{{R-{\rm tf}}}
\nc{\dtor}{{R-{\rm tor}}}

\nc{\Div}{{\mrm Div}}
\nc{\Diff}{\mrm{DA}}
\nc{\Diffl}{\mathsf{DA}_\lambda}
\nc{\diffo}{{\mathsf{DO}_\lambda}}
\nc{\dl}{{\mathrm{PD}}}
\nc{\dRB}{{\mathrm{\Phi}_\mathsf{DRB}}}
\nc{\udRB}{{\mathrm{\Phi}_\mathsf{uDRB}}}
\nc{\OdRB}{{\mathrm{\Phi}_\mathsf{DRB}^0}}
 \nc{\OudRB}{{\mathrm{\Phi}_\mathsf{uDRB}^0}}
\nc{\inte}{{\mathrm{\Phi}_\mathsf{ID}}}
\nc{\uinte}{{\mathrm{\Phi}_\mathsf{uID}}}
\nc{\Ointe}{{\mathrm{\Phi}_\mathsf{ID}^0}}
 \nc{\Ouinte}{{\mathrm{\Phi}_\mathsf{uID}^0}}

\nc{\alg}{\mathsf{Alg}}
\nc{\End}{\mrm{End}}
\nc{\Ext}{\mrm{Ext}}
\nc{\Fil}{\mrm{Fil}}
\nc{\Fr}{\mrm{Fr}}
\nc{\Frob}{\mrm{Frob}}
\nc{\Gal}{\mrm{Gal}}
\nc{\GL}{\mrm{GL}}
\nc{\Hom}{\mrm{Hom}}
\nc{\Hoch}{\mrm{Hoch}}
\nc{\hsr}{\mrm{H}}
\nc{\hpol}{\mrm{HP}}
\nc{\id}{\mrm{id}}
\nc{\im}{\mrm{im}}
\nc{\Id}{\mrm{Id}}
\nc{\ID}{\mrm{ID}}
\nc{\Irr}{\mrm{Irr}}
\nc{\incl}{\mrm{incl}}
\nc{\length}{\mrm{length}}
\nc{\NLSW}{\mrm{NLSW}}
\nc{\Lie}{\mrm{Lie}}
\nc{\Nij}{\mrm{Nij}}
\nc{\mchar}{\rm char}
\nc{\mpart}{\mrm{part}}
\nc{\ql}{{\QQ_\ell}}
\nc{\qp}{{\QQ_p}}
\nc{\rank}{\mrm{rank}}
\nc{\rcot}{\mrm{cot}}
\nc{\rdef}{\mrm{def}}
\nc{\rdiv}{{\rm div}}
\nc{\Rey}{\mrm{Rey}}
\nc{\rtf}{{\rm tf}}
\nc{\rtor}{{\rm tor}}
\nc{\res}{\mrm{res}}
\nc{\SL}{\mrm{SL}}
\nc{\Spec}{\mrm{Spec}}
\nc{\tor}{\mrm{tor}}
\nc{\Tr}{\mrm{Tr}}
\nc{\tr}{\mrm{tr}}
\nc{\wt}{\mrm{wt}}
\def\ot{\otimes}

\nc{\udl}{{\mathrm{udl}}}

\nc{\bfk}{{\bf k}}
\nc{\bfone}{{\bf 1}}
\nc{\bfzero}{{\bf 0}}
\nc{\detail}{\marginpar{\bf More detail}
	\noindent{\bf Need more detail!}
	\svp}
\nc{\gap}{\marginpar{\bf Incomplete}\noindent{\bf Incomplete!!}
	\svp}
\nc{\FMod}{\mathbf{FMod}}
\nc{\Int}{\mathbf{Int}}
\nc{\remarks}{\noindent{\bf Remarks: }}
\nc{\ob}{\mathsf{Ob}}

\nc{\BA}{{\mathbb A}}   \nc{\CC}{{\mathbb C}}
\nc{\DD}{{\mathbb D}}   \nc{\EE}{{\mathbb E}}
\nc{\FF}{{\mathbb F}}   \nc{\GG}{{\mathbb G}}
\nc{\HH}{{\mathbb H}}   \nc{\LL}{{\mathbb L}}
\nc{\NN}{{\mathbb N}}   \nc{\PP}{{\mathbb P}}
\nc{\QQ}{{\mathbb Q}}   \nc{\RR}{{\mathbb R}}
\nc{\TT}{{\mathbb T}}   \nc{\VV}{{\mathbb V}}
\nc{\ZZ}{{\mathbb Z}}   \nc{\TP}{\widetilde{P}}

\nc{\cala}{{\mathcal A}}    \nc{\calc}{{\mathcal C}}
\nc{\cald}{\mathcal{D}}     \nc{\cale}{{\mathcal E}}
\nc{\calf}{{\mathcal F}}    \nc{\calg}{{\mathcal G}}
\nc{\calh}{{\mathcal H}}    \nc{\cali}{{\mathcal I}}
\nc{\call}{{\mathcal L}}    \nc{\calm}{{\mathcal M}}
\nc{\caln}{{\mathcal N}}    \nc{\calo}{{\mathcal O}}
\nc{\calp}{{\mathcal P}}    \nc{\calr}{{\mathcal R}}
\nc{\cals}{{\mathcal S}}    \nc{\calt}{{\Omega}}
\nc{\calv}{{\mathcal V}}    \nc{\calw}{{\mathcal W}}
\nc{\calx}{{\mathcal X}}    \nc{\calu}{{\mathcal U}}
\nc{\caly}{{\mathcal Y}}

\nc{\uOpAlg}{{\mathfrak{uOpAlg}}}

\nc{\OpAlg}{{\mathfrak{OpAlg}}}

\nc{\ComOpAlg}{{\mathfrak{ComOpAlg}}}

\nc{\OpVect}{{\mathfrak{OpVect}}}

\nc{\OpSet}{{\mathfrak{OpSet}}}

\nc{\OpMon}{{\mathfrak{OpMon}}}

\nc{\ComOpMon}{{\mathfrak{ComOpMon}}}

\nc{\OpSem}{{\mathfrak{OpSem}}}

\nc{\ComOpSem}{{\mathfrak{ComOpSem}}}

\nc{\uAlg}{{\mathfrak{uAlg}}}

\nc{\Alg}{{\mathfrak{Alg}}}

\nc{\ComAlg}{{\mathfrak{ComAlg}}}

\nc{\Vect}{{\mathfrak{Vect}}}

\nc{\Set}{{\mathfrak{Set}}}

\nc{\Mon}{{\mathfrak{Mon}}}

\nc{\ComMon}{{\mathfrak{ComMon}}}

\nc{\Sem}{{\mathfrak{Sem}}}

\nc{\mtOpSet}{{\Omega\zhx\mathfrak{Set}}}

\nc{\mtOpSem}{{\Omega\zhx\mathfrak{Sem}}}

\nc{\mtOpMon}{{\Omega\zhx\mathfrak{Mon}}}

\nc{\mtOpVect}{{\Omega\zhx\mathfrak{Vect}}}

\nc{\mtOpAlg}{{\Omega\zhx\mathfrak{Alg}}}

\nc{\mtuOpAlg}{{\Omega\zhx\mathfrak{uAlg}}}

\nc{\ComSem}{\mathfrak{ComSem}}
\nc{\fraka}{{\mathfrak a}}
\nc{\frakb}{\mathfrak{b}}
\nc{\frakg}{{\frak g}}
\nc{\frakl}{{\frak l}}
\nc{\fraks}{{\frak s}}
\nc{\frakB}{{\frak B}}
\nc{\frakm}{{\frak m}}
\nc{\frakM}{{\mathfrak{M}_\Omega}}
\nc{\frakp}{{\frak p}}
\nc{\frakW}{{\frak W}}
\nc{\frakX}{{\frak X}}
\nc{\frakS}{{\frak{S}_\Omega}}
\nc{\frakA}{{\frak A}}
\nc{\frakx}{{\frakx}}

\nc{\frakMstar}{{\mathfrak{M}_\Omega^\star}}
\nc{\frakSstar}{{\mathfrak{S}_\Omega^\star}}

\nc{\lir}[1]{\textcolor{red}{\underline{Li:}#1 }}

\begin{document}

\title[GS bases for free multi-operated algebras over algebras]{Gr\"obner-Shirshov bases and linear bases    for  free multi-operated algebras over algebras with applications to differential Rota-Baxter algebras and   integro-differential algebras}

\author{Zuan Liu, Zihao Qi, Yufei Qin and Guodong Zhou}

\address{Zuan Liu, Yufei Qin and Guodong Zhou,  School of Mathematical Sciences, Shanghai Key Laboratory of PMMP,
	East China Normal University,
	Shanghai 200241,
	China}
\email{2451533837@qq.com}
\email{290673049@qq.com}
 \email{gdzhou@math.ecnu.edu.cn}

 \address{Zihao  Qi, School of Mathematical Sciences,
	Fudan University, Shanghai 200433, China}
\email{qizihao@foxmail.com}

\date{\today}

\begin{abstract}  Quite much recent studies has been attracted to the operated algebra since it unifies various notions such as the differential algebra and the Rota-Baxter algebra. An $\Omega$-operated algebra is a an (associative) algebra equipped with a set $\Omega$ of linear operators which might satisfy certain operator identities such as the Leibniz rule. A free $\Omega$-operated algebra $B$ can be generated on an algebra $A$ similar to a free algebra generated on a set. If $A$ has a Gr\"{o}bner-Shirshov basis $G$ and if the linear operators $\Omega$ satisfy a set $\Phi$ of operator identities, it is natural to ask when the union $G\cup \Phi$ is a Gr\"{o}bner-Shirshov basis of $B$. A previous work answers this question affirmatively under a mild condition, and thereby obtains a canonical linear basis of $B$.

In this paper, we answer this question in the general case of multiple linear operators. As applications we get operated Gr\"{o}bner-Shirshov bases for free differential Rota-Baxter algebras and free integro-differential algebras over algebras as well as their linear bases. One of the key technical difficulties is to introduce new monomial orders for the case of two operators, which might be of independent interest.

\end{abstract}

\subjclass[2010]{
16Z10 
03C05 
08B20 
12H05 
16S10 
17B38  
}

\keywords{}

\maketitle

 \tableofcontents

\allowdisplaybreaks

\section*{Introduction}

  This paper extends the results of \cite{QQWZ21} to  algebras endowed with several operators, with applications to differential Rota-Baxter algebras and   integro-differential algebras.

 \subsection{Operated GS basis theory: from a single operator to multiple operators}\

Since its introduction  by Shirshov \cite{Shirshov} and Buchberger \cite{Buc} in the  sixties  of last century,   Gr\"obner-Shirshov (=GS) basis theory  has become one of the main tools of computational algebra; see for instance \cite{Green, BokutChen14, BokutChen20}. In order to deal with algebras endowed with operators, Guo and his coauthors introduced   a  GS basis theory   in a series of papers \cite{Guo09a,  GGSZ14, GSZ13,  GaoGuo17} (see also \cite{BokutChenQiu})
 with the goal to attack Rota's program \cite{Rota} to classify ``interesting'' operators on algebras.
Guo et al.  considered operators satisfying some polynomial identities, hence called operated polynomial identities (aka. OPIs) \cite{Guo09a,  GGSZ14, GSZ13,  GaoGuo17}. Via GS basis theory and  the somewhat equivalent theory: rewriting systems, they could define when OPIs are GS.  They are mainly interested into two classes of OPIs:  differential type OPIs and Rota-Baxter type OPIs, which are carefully  studied  in \cite{GSZ13, GGSZ14, GaoGuo17}.
 For   the state of art, we refer the reader to the survey paper \cite{GGZ21}  and for recent development, see \cite{ZhangGaoGuo21,  GuoLi21,  QQWZ21,   ZhangGaoGuo}.

 In these  papers \cite{Guo09a,  GGSZ14, GSZ13,  GaoGuo17}, the operated GS theory and hence Rota's classification  program have been  carried out only   for  algebras endowed with a single operator.
 It would be very interesting to carry out further   Rota's program for  the general case of multiple linear operators.

The paper \cite{BokutChenQiu}   contains a first step of this program  by   developing the GS basis theory in this generalised setup. We will review and update the GS basis theory in the multi-operated setup in Section~\ref{Section: multi-operated  GS bases}.
 
 Another direction is to generalise from  operated algebras over a base field to operated algebras over a base ring. While previous papers  \cite{QQWZ21, QQZ21} considered this aspect for single operator case, this paper  is aimed to deal with this aspect for     multiple linear operator case. In particular, some new monomial orders for the two operator case will be constructed which enable us to study operated GS bases for free  operated algebras generated by   algebras, while it seems that the monomial orders appeared in previous papers can be applied directly when the base ring is not a field any more.


\medskip

\subsection{Free operated algebras over algebras}\

Recently, there is a need to develop free  operated algebras satisfying some OPIs  over a fixed  algebras and construct   GS  bases and linear bases  for these free  algebras as long as a GS  basis is known for the given algebra.   Ebrahimi-Fard and Guo  \cite{ERG08a} used rooted trees and forests to give explicit constructions of free noncommutative
Rota-Baxter algebras on modules and sets; Lei and Guo \cite{LeiGuo} constructed the linear basis of free Nijenhuis algebras over associative algebras;
   Guo and Li \cite{GuoLi21} gave a linear basis of the free differential algebra over associative algebras by introducing the notion of differential GS bases.

In a previous paper \cite{QQWZ21}, the authors considered a  question which can be roughly stated as follows:
\begin{Ques} \label{Ques: GS basis for free algebras over algebras} Given a (unital or nonunital)  algebra  $A$  with a   GS basis  $G$ and  a set $\Phi$ of OPIs,
	  assume that   these OPIs  $\Phi$ are GS  in the sense of \cite{BokutChenQiu, GSZ13, GGSZ14,GaoGuo17}.  Let $B$
be the free  operated algebra satisfying $\Phi$ over $A$.  When will $\Phi\cup G$  be  a GS  basis  for $    B$?
\end{Ques}
They answer this question in the affirmative under a mild condition in \cite[Theorem 5.9]{QQWZ21}.  When this   condition is satisfied, $\Phi\cup G$  is  a GS  basis  for $    B$ and as a consequence, we also get  a linear basis of $B$. This result    has been applied to all Rota-Baxter type OPIs, a class of differential type OPIs, averaging OPIs and Reynolds OPI in \cite{QQWZ21}.
It was also applied to
  differential type OPIs by  introducing
  some new monomial orders    \cite{QQZ21}.

  In this paper, we consider a similar question for multi-operated algebras.

  Let $\Omega$ be a nonempty set which will be the index set of operators. Algebras endowed with operators indexed by $\Omega$ will be called $\Omega$-algebras.  OPIs can be extended to the multi-operated setup and one can introduce the notion of $\Omega$-GS for OPIs.

 \begin{Ques}  \label{Ques: nulti-operated setting}  Let $\Phi$ be  a set  of OPIs of a set of operators indexed by $\Omega$.  Let $A$ be  a (unital) algebra  together   with a   GS  basis $G$.  Assume that   these OPIs  $\Phi$ are GS  in the sense of  Section~\ref{Section: multi-operated  GS bases}.  Let $B$
be the free  $\Omega$-algebra over $A$ such that  the operators satisfy  $\Phi$.  When will $\Phi\cup G$ be  an $\Omega$-GS  basis  for $    B$?
\end{Ques}

We extend  the main result   of \cite{QQWZ21} to multi-operated cases; see Theorem~\ref{Thm: GS basis for free unital Phi algebra over unital alg} for unital algebras  and Theorem~\ref{Thm: GS basis for free nonunital Phi algebra over nonunital alg} for nonunital algebras.

\medskip

\subsection{Differential Rota-Baxter algebras and   integro-differential algebras}\

The main motivation of this paper comes, in fact,  from differential Rota-Baxter algebras and   integro-differential algebras.

Differential Rota-Baxter algebras were introduced by Guo and Keigher \cite{GK08} which  reflect  the relation between the differential operator and the integral
operator as in the First Fundamental Theorem of Calculus.  Free differential Rota-Baxter algebras were constructed by using various tools including  angularly decorated rooted forests  and GS basis theory \cite{GK08, BokutChenQiu}.

Integro-differential algebras (of zero weight) were defined  for the algebraic study of boundary problems
for  linear systems of linear ordinary differential equations.  Guo, Regensburger and Rosenkranz \cite{GRR14} introduced Integro-differential algebras with weight.
Free objects and their linear bases   were constructed by using GS basis theory \cite{GRR14, GGZ14, GGR15}

The main goal of this paper is to study free differential Rota-Baxter algebras  and free integro-differential algebras  over algebras from the viewpoint of operated GS basis theory.
In particular, when the base algebra is reduced to $\bfk$, our results also give GS bases and linear bases   for free differential Rota-Baxter algebras  and free integro-differential algebras.

However, the original monomial orders used in \cite{BokutChenQiu, GRR14, GGZ14, GGR15}  do not satisfy the hypothesis in Theorems~\ref{Thm: GS basis for free unital Phi algebra over unital alg} and  \ref{Thm: GS basis for free nonunital Phi algebra over nonunital alg} for free multi-operated algebras over algebras, and we have to introduce a new monomial order  $\leq_{\operatorname{PD}}$ (resp.~$\leq_{\operatorname{uPD}}$) to overcome the problem; see Section~\ref{section: order}.

In contrast to  the use  different monomial orders while dealing with free differential Rota-Baxter algebras and free integro-differential algebras  in \cite{BokutChenQiu}  and     \cite{GGR15} respectively, we will demonstrate that our monomial ordering  $\leq_{\operatorname{PD}}$       can be applied to both types of algebras   simultaneously, as we shall see in Sections~\ref{Section: differential Rota-Baxter algebras} and \ref{Section: integro-differential algebras}.
Moreover, since the case of the unital algebras was not discussed in \cite{BokutChenQiu}, this aspect is addressed in  Subsection~\ref{Subsection: case unital   algebras} by using our monomial order  $\leq_{\operatorname{uPD}} $.

\medskip

\subsection{Outline of the paper}\

This paper is organized as follows.

The first section contains remainder on  free objects in multi-operated setting and on the construction of free $\Omega$-semigroups and related structures,
 and introduces some new monomial orders for the case of two operators, which will be the key technical tool of this paper.

 In the second section, we recall the theory of GS bases for the multi-operated setting.  After introducing OPIs,  GS property for OPIs and $\Omega$-GS bases for multi-operated algebras are defined; after giving some facts about free multi-operated $\Phi$-algebras on algebras, answers to Question~\ref{Ques: nulti-operated setting} are presented.

 In the third section, multi-operated GS bases and linear bases   for free differential Rota-Baxter algebras on algebras are studied and the fourth section contains our investigation  for free   integro-differential algebras on algebras.

\medskip

\textbf{Notation:} Throughout this paper, $\bk$ denotes a  base field. All the vector spaces and algebras are over $\bk$.

\bigskip

\section{New monomial orders on  free multi-operated semigroups and monoids}

In this section,  we recall  free objects in multi-operated setting and the construction of free $\Omega$-semigroups and related structures, and
   define two new monomial  orders  $\leq_{\operatorname{PD}}$ and $\leq_{\operatorname{uPD}}$  on free multi-operated semigroups and monoids.
    The  main results of this paper will  highly depend on these new monomial  orders.

  For a set $Z$, denote by $\bk Z$ (resp.  $\cals(Z)$, $\calm(Z)$) the free $\bk$-vector space (resp. free semigroup, free monoid) generated by $Z$. Denote the category of  sets (resp. semigroups, monoids)  by $\Set$ (resp. $\Sem$, $\Mon$). Denote the categories of  $\bk$-algebras and unital $\bk$-algebras by $\Alg$ and $\uAlg$  respectively.

Throughout this section, let $\Omega$ be a nonempty set which will be the index set of operators.

\subsection{Free objects in the multi-operated setup}\

\begin{defn}
  An operated set with an operator index set $\Omega$ or simply  an $\Omega$-set is a set $S$ endowed with a family of maps
    $P_\omega: S\rightarrow S$ indexed by $\omega\in \Omega$.  The morphisms between $\Omega$-sets  can be defined in the obvious way.    Denote the category of  $\Omega$-sets by $\mtOpSet$.

  Similarly, we can define  $\Omega$-semigroups and $\Omega$-monoids.  Their categories are denoted by   $\mtOpSem$ and  $\mtOpMon$ respectively.

$\Omega$-vector spaces,   nonunital or unital  $\Omega$-algebras can be defined in a similar way,  except asking, moreover, that  all the operators are $\bk$-linear maps. Denote the category of $\Omega$-vector spaces, (resp. nonunital $\Omega$-algebras,  unital  $\Omega$-algebras) by $\mtOpVect$ (resp.  $\mtOpAlg$, $\mtuOpAlg$) with obvious morphisms.
\end{defn}
%
%
%

As in  \cite{QQWZ21}, there exists the following diagram of  functors:
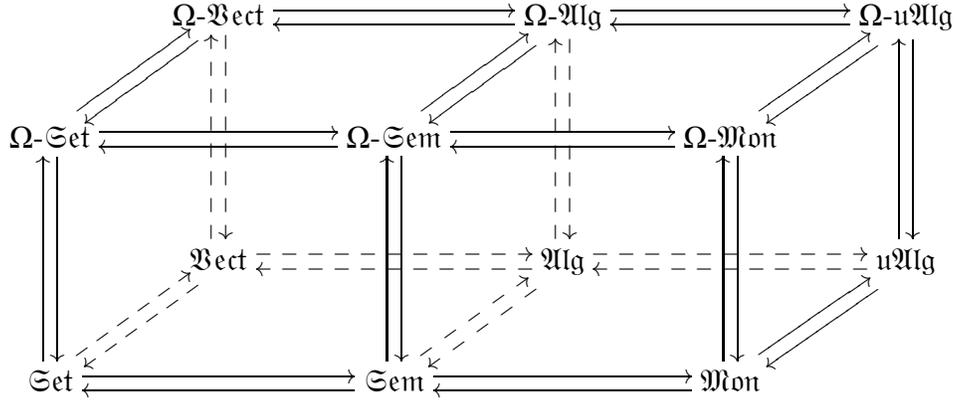
\[\xymatrixrowsep{2.5pc}
\xymatrix{
	& \mtOpVect \ar@<.5ex>@{->}[rr]\ar@<.5ex>@{->}[ld]\ar@<.5ex>@{-->}[dd]
	& &\mtOpAlg   \ar@<.5ex>@{->}[rr]\ar@<.5ex>@{->}[ll]\ar@<.5ex>@{->}[ld]\ar@<.5ex>@{-->}[dd]
	&  & \mtuOpAlg \ar@<.5ex>@{->}[ld] \ar@<.5ex>@{->}[ll]\ar@<.5ex>@{->}[dd] \\
	\mtOpSet \ar@<.5ex>@{->}[ur] \ar@<.5ex>@{->}[rr] \ar@<.5ex>@{->}[dd]
	& &\mtOpSem  \ar@<.5ex>@{->}[ur]   \ar@<.5ex>@{->}[rr]\ar@<.5ex>@{->}[ll]\ar@<.5ex>@{->}[dd]
	& & \mtOpMon  \ar@<.5ex>@{->}[ur]\ar@<.5ex>@{->}[ll]  \ar@<.5ex>@{->}[dd]  &\\
	&  \Vect  \ar@<.5ex>@{-->}[uu]   \ar@<.5ex>@{-->}[rr] \ar@<.5ex>@{-->}[ld]
	&  &  \Alg   \ar@<.5ex>@{-->}[uu]  \ar@<.5ex>@{-->}[rr]\ar@<.5ex>@{-->}[ll]\ar@<.5ex>@{-->}[ld]
	& &  \uAlg  \ar@<.5ex>@{->}[uu]\ar@<.5ex>@{-->}[ll]\ar@<.5ex>@{->}[ld]   \\
	\Set \ar@<.5ex>@{->}[uu] \ar@<.5ex>@{-->}[ur] \ar@<.5ex>@{->}[rr]
	&  & \Sem   \ar@<.5ex>@{->}[rr] \ar@<.5ex>@{->}[uu]\ar@<.5ex>@{-->}[ur]\ar@<.5ex>@{->}[ll]
	&  & \Mon  \ar@<.5ex>@{->}[ur]  \ar@<.5ex>@{->}[uu] \ar@<.5ex>@{->}[ll] &   }
\]
In this diagram, all functors from right to left, from below to above and from southwest to northeast are the obvious forgetful functors. The other functors are free object functors   which  are left adjoint to    the   forgetful functors.

Our notations for free object functors are analogous to those in \cite{QQWZ21}. For instance, $\calf^{\mtOpAlg}_{\Alg}$ denotes the free object functor from
the category of algebras to that of nonunital $\Omega$-algebras.

We could give similar constructions of these free object functors as in Sections 1-3 of \cite{QQWZ21}. However, as we don't need the details, we will not repeat them. The curious readers could consult \cite{QQWZ21} and extend the constructions in \cite{QQWZ21} without essential difficulties.

\subsection{Free multi-operated semigroups and monoids}\


%
%


  Now we explain  the construction of the  free $\Omega$-semigroup generated by  a set $Z$.

For $\omega\in \Omega$, denote by  $\lfloor Z \rfloor_\omega$     the set of all formal elements $ \lfloor z \rfloor_\omega, z\in Z$ and put
$\lfloor Z \rfloor_\Omega=\sqcup_{\omega\in \Omega} \left \lfloor Z\right \rfloor_\omega$. The inclusion into the first component  $ Z\hookrightarrow Z\sqcup \lfloor Z \rfloor_\Omega$ induces an injective semigroup homomorphism
$$ i_{0,1}: \frakS_{,0}(Z):=\cals(Z)  \hookrightarrow \frakS_{,1}(Z):= \cals(Z\sqcup  \lfloor Z \rfloor_\Omega).
$$

  For $n \geq 2$, assume  that we have constructed $\frakS_{,n-2}(Z)$ and  $\frakS_{,n-1}(Z)= \cals(Z\sqcup \lfloor \frakS_{,n-2}(Z) \rfloor_\Omega)$  endowed  with  an  injective homomorphism of semigroups
 $
i_{n-2,n-1}: \frakS_{,n-2}(Z) \hookrightarrow \frakS_{,n-1}(Z).
$
We define the semigroup
$$
\frakS_{,n}(Z):= \cals(Z\sqcup \lfloor \frakS_{,n-1}(Z) \rfloor_\Omega)
$$ and
the natural  injection
$$
\Id_Z\sqcup \lfloor i_{n-2,n-1} \rfloor_\Omega:Z\sqcup \lfloor \frakS_{,n-2}(Z) \rfloor_\Omega \hookrightarrow Z\sqcup \lfloor \frakS_{,n-1}(Z) \rfloor_\Omega
$$
induces an injective semigroup  homomorphism
$$	i_{n-1,n}: \frakS_{,n-1}(Z)= \cals(Z\sqcup \lfloor  \frakS_{,n-2}(Z) \rfloor_\Omega) \hookrightarrow  \frakS_{,n}(Z) = \cals (Z\sqcup \lfloor \frakS_{,n-1}(Z) \rfloor_\Omega).$$
Define $  \frakS(Z)=\varinjlim \frakS_{,n}(Z) $  and the maps  sending  $u\in \frakS_{,n}(Z)$ to $\left \lfloor u\right \rfloor_{\omega}\in \frakS_{,n+1}(Z)$ induces a family of  operators $P_{\omega}, \omega \in \Omega$  on $\frakS(Z)$.


The construction of the free $\Omega$-monoid $\frakM (M)$ over a set $Z$ is similar, by just replacing $\cals(Z)$ by $\calm(Z)$ everywhere in the construction.

 \begin{remark}\label{remark: monoids}
 We will use another  construction of  $\frakM (Z)$.  In fact, add some symbols $\lfloor1\rfloor_\Omega=\{\lfloor1\rfloor_\omega, \omega\in \Omega\}$  to $Z$ and form $ \frakS (Z \sqcup \lfloor1\rfloor_\Omega  )$, then $\frakM (Z)$ can be obtained from $  \frakS(Z \sqcup \lfloor1\rfloor_\Omega  )$ by just adding the empty word $1$.

 \end{remark}

 It is easy to see that $\bk\frakS(Z)$(resp. $\bk\frakM (Z)$) is the free nonunital (resp. unital) $\Omega$-algebra generated by $Z$.


\subsection{Monomial orders}\label{section: order}\

In this subsection, we introduce some new monomial orders on   free $\Omega$-semigroups and  free $\Omega$-monoids. We
    only consider the case of two operators, say $\Omega=\{P, D\}$ as the main examples in mind are differential Rota-Baxter algebras and   integro-differential algebras following the convention from \cite{GGR15}.

     We first recall the definitions of well orders and monomial orders.

\begin{defn}
	Let $Z$ be a nonempty set.
	\begin{itemize}
		\item [(a)] A preorder   $\leq$  is a binary relation on $Z$ that is reflexive and transitive, that is, for all $x, y, z \in Z$, we have
		\begin{itemize}
			\item [(i)] $x \leq x$; and
			\item[(ii)]if $x \leq y, y \leq z$, then $x \leq z$.
		\end{itemize}
 In the presence of a preoder $\leq$, we denote $x=_{Z} y$ if $x \leq y$ and $x \geq y$; if $x \leq y$ but $x \neq y$, we write $x< y$ or $y> x$.
		\item[(b)] A pre-linear order $\leq$ on $Z$ is a preorder $\leq$ such that either $x \leq y$ or $x \geq y$ for all $x, y \in Z$.
		\item[(c)] A linear order or a total order $\leq$ on $Z$ is a pre-linear order $\leq$  such that  $\leq$ is symmetric, that is,   $x \leq y$ and $y \leq x$ imply    $x=y$.

\item[(d)] A preorder   $\leq$ on $Z$ is said to satisfy  the descending chain condition, if for each descending chain $x_1\geq x_2\geq x_3\geq \cdots$,   there exists $N\geq 1$ such that $x_N=_Z x_{N+1}=_Z\cdots$.
    A linear order satisfying the descending chain condition is called a well order.

	\end{itemize}
\end{defn}

Before giving the definition of monomial orders, we need to introduce the following notions generalising the case of one operator.
\begin{defn}
	Let $Z$ be a set and $\star$ a symbol not in $Z$.
	\begin{itemize}
		\item [(a)] Define $\frakMstar(Z)$  to be the subset of $\frakM(Z\cup\star)$ consisting of elements with $\star$ occurring only once.
		\item [(b)] For $q\in\frakMstar(Z)$ and $u\in   \frakM(Z)$,	we define $q|_{u}\in \frakM(Z)$ to be the element obtained by
		replacing the symbol $\star$ in $q$ by $u$. In this case, we say $u$ is a subword of $q|_{u}$.
		\item [(c)] For $q\in\frakMstar(Z)$ and $s=\sum_ic_iu_i\in \bk \frakM(Z)$  with $c_i\in\bk$ and $u_i\in\frakM(Z)$, we define
		$$q|_s:=\sum_ic_iq|_{u_i}.$$
		\item [(d)] Define $\frakSstar(Z)$ to be the subset of $\frakS(Z\cup\star)$ consisting of elements with $\star$ occurring only once. It is easy to see $\frakSstar(Z)$ is a subset of $\frakMstar(Z)$, so we also have notations in (a)-(c) for $\frakSstar(Z)$ by restriction.
	\end{itemize}
\end{defn}

\begin{defn}
	Let $Z$ be a set. 
	
	\begin{itemize}	
		\item [(a)]A monomial order on $\cals(Z)$ is a well-order $\leq$ on $\cals(Z)$ such that
		$$  u < v \Rightarrow  uw < vw\text{ and }wu<wv  \text{ for any }u, v, w\in \cals(Z);$$
		\item [(a')] a monomial order on $\calm(Z)$ is a well-order $\leq$ on $\calm(Z)$ such that
		$$ u < v \Rightarrow  wuz < wvz \text{ for any }u, v, w,z\in \calm(Z);$$
		
		\item [(b)]a monomial order on  $\frakS(Z)$  is a well-order $\leq$ on $\frakS(Z)$  such that
		$$u< v \Rightarrow q|_u<q|_v\quad\text{for all }u,v\in\frakS(Z)\text{ and } q\in \frakSstar(Z);$$
		\item [(b')]a monomial order on $\frakM(Z)$ is a well-order $\leq$ on $\frakM(Z)$ such that
		$$u< v \Rightarrow q|_u<q|_v\quad\text{for all }u,v\in\frakM(Z)\text{ and } q\in \frakMstar(Z). $$
	\end{itemize}
\end{defn}

Let us recall some known preorders.

\begin{defn}
	
	For two elements $u, v \in \frakS(Z)$,
	\begin{itemize}
		\item [(a)]  define
		$$
		u \leq_{\operatorname{D}} v \Leftrightarrow \operatorname{deg}_{D}(u) \leq \operatorname{deg}_{D}(v),
		$$
		where the $D$-degree $\operatorname{deg}_{D}(u)$ of $u$ is the number of occurrence  of $\lfloor~\rfloor_D$ in $u$;
           \item [(b)]  define
		$$
		u \leq_{\operatorname{P}} v \Leftrightarrow \operatorname{deg}_{P}(u) \leq \operatorname{deg}_{P}(v),
		$$
		where the $P$-degree $\operatorname{deg}_{P}(u)$ of $u$ is the number of occurrence  of $\lfloor~\rfloor_P$ in $u$;

		\item [(c)]   define
		$$
		u \leq_{\operatorname{dZ}} v \Leftrightarrow \operatorname{deg}_{Z}(u) \leq \operatorname{deg}_{Z}(v),
		$$
		where the $Z$-degree $\operatorname{deg}_{Z}(u)$ is the number of elements of $Z$ occurring in $u$ counting  the repetitions;
	\end{itemize}
	
\end{defn}

\begin{defn}\label{Def: deg-lex order}
	Let $Z$ be a set endowed with a well order $\leq_{Z}$.   
	Introduce  the degree-lexicographical order $\leq_{\rm {dlex }}$ on $\cals(Z)$ by imposing, for any $u\neq  v \in \cals(Z)$, $u <_{\rm {dlex}}v$   if
	\begin{itemize}
		\item[(a)] either $\operatorname{deg}_{Z}(u)<\operatorname{deg}_{Z}(v)$, or
		\item[(b)] $\operatorname{deg}_{Z}(u)=\operatorname{deg}_{Z}(v)$, and $u=mu_{i}n$, $v=mv_{i}n^\prime$ for some $m,n,n^\prime\in \calm(Z)$ and $u_{i},v_{i}\in Z$ with $u_{i}<_{Z} v_{i}$.
	\end{itemize}
\end{defn}
It  is obvious that the degree-lexicographic order $\leq_{\mathrm{dlex}}$ on $\cals(Z)$ is a well order .


  We now define a preorder $\leq_{\operatorname{Dlex}}$ on $\frakS(Z)$, by the following recursion process:
\begin{itemize}
	\item [(a)] For  $u,v\in \frakS_{,0}(Z)=\cals(Z)$,  define
	$$u\leq_{\operatorname{Dlex}_0} v \Leftrightarrow u \leq_{\mathrm{dlex}}v.$$
	\item [(b)] Assume that we have constructed  a well order $\leq_{\operatorname{Dlex}_n}$  on $\frakS_{,n}(Z)$ for $n\geq 0$ extending all $\leq_{\operatorname{Dlex}_i}$ for any $0\leq i\leq n-1$.   The  well order $\leq_{\operatorname{Dlex}_n}$ on $\frakS_{,n}(Z)$ induces a well order   on $\lfloor\frakS_{,n}(Z)\rfloor_P$  (resp. $\lfloor\frakS_{,n}(Z)\rfloor_D$), by imposing
$\lfloor u\rfloor_P\leq \lfloor v\rfloor_P$  (resp.  $\lfloor u\rfloor_D\leq \lfloor v\rfloor_D$) whenever $u\leq_{\operatorname{Dlex}_n} v\in \frakS_{,n}(Z)$.
  By setting   $u<v<w$ for all  $u\in Z$, $v\in \lfloor\frakS_{,n}(Z)\rfloor_D$, and $w\in \lfloor\frakS_{,n}(Z)\rfloor_P$, we obtain a  well order  on  $Z\sqcup \lfloor\frakS_{,n}(Z)\rfloor_P \sqcup \lfloor\frakS_{,n}(Z)\rfloor_D$.
Let    $\leq_{\operatorname{Dlex}_{n+1}}$ be  the  degree lexicographic order on   $\frakS_{,n+1}(Z)=\cals(Z\sqcup \lfloor\frakS_{,n}(Z)\rfloor_P \sqcup \lfloor\frakS_{,n}(Z)\rfloor_D)$  induced by that on $Z\sqcup \lfloor\frakS_{,n}(Z)\rfloor_P \sqcup \lfloor\frakS_{,n}(Z)\rfloor_D$.
\end{itemize}
Obviously     $\leq_{\operatorname{Dlex}_{n+1}}$ extends    $\leq_{\operatorname{Dlex}_{n}}$. By a limit process, we get a preorder on $\frakS(Z)$ which will be  denoted by
 $\leq_{\operatorname{Dlex}}$. As is readily seen, $\leq_{\operatorname{Dlex}}$ is a linear order.

\begin{remark}\label{remark: monomial order for multi-operator case}
	It is easy to see that the above  construction of $\leq_{\operatorname{Dlex}}$ can be  extended to the case of more than two operators.

In fact, for  a given well order
$\leq_{\Omega}$ in the index set  $\Omega$,
the defining process of   $\leq_{\operatorname{Dlex}}$ on $\frakS(Z)$ is the same as above except one detail in step (b), where
we need to put  $u<v<w$ for all  $u\in Z$, $v\in \lfloor\frakS_{,n}(Z)\rfloor_{\omega_1}$ and $w\in \lfloor\frakS_{,n}(Z)\rfloor_{\omega_2}$    with  $\omega_1 \leq_{\Omega} \omega_2\in \Omega$.
\end{remark}

\begin{defn}\label{GD}

For any $u\in\frakS(Z)$, let $u_1,\dots ,u_n\in Z$ be  all the elements occurring in $u$ from left to right. If a right half bracket $\rfloor_D$ locates in the gap between $u_i$ and $u_{i+1}$, where $1\leq i<n$,   the  GD-degree of this right half bracket is defined  to be $n-i$;  if there is a  right half bracket $\rfloor_D$ appearing  on the right of $u_n$, we define the GD-degree of this half bracket to be  $0$. We define the GD-degree of $u$, denoted by $\deg_{GD}(u)$, to be  the sum of the   GD-degrees  of all the half right brackets in $u$.

For example,  the GD-degrees of the half right brackets in  $u=\left \lfloor x \right \rfloor_D \left \lfloor y \right \rfloor_D$   with  $x,y \in Z$ are  respectively 1 and 0 from left to right, so   $\deg_{GD}(u)=1$ by definition.

For $u, v\in\frakS(Z)$, define the GD-degree  order $\leq_{\mathrm{GD}}$  by
         $$ u\leq_{\mathrm{GD}}v\Leftrightarrow \deg_{GD}(u)\leq \deg_{GD}(v).$$

\end{defn}

\begin{defn}\label{GP}
For any $u\in\frakS(Z)$, let $u_1,\dots ,u_n\in Z$ be  all the elements occurring in $u$ from left to right. If there are $i$ elements in $Z$ contained in a bracket $\lfloor~\rfloor_P$ ,   the  GP-degree of this  bracket is defined  to be $n-i$. We denote by $\deg_{GP}(u)$   the sum of the  GP-degree  of all the  brackets  $\lfloor~\rfloor_P$ in $u$.

For example, the the GP-degrees of the brackets $\lfloor~\rfloor_P$ of $u=\left \lfloor xy \right \rfloor_P \left \lfloor z \right \rfloor_P$ with $x,y,z \in Z$ are respectively 1 and 2 from left to right, so  $\deg_{GP}(u)=3$ by definition.

For $u,v\in\frakS(Z)$, define the GD-degree  order $\leq_{\mathrm{GD}}$  by
         $$ u\leq_{\mathrm{GP}}v\Leftrightarrow \deg_{GP}(u)\leq \deg_{GP}(v).$$

\end{defn}

It is easy to obtain the following lemma whose  proof is thus omitted.

\begin{lem}\label{lemma: prelinear order with descending chain	condition}
	The orders $\leq_{\mathrm{D}}$, $\leq_{\mathrm{P}}$, $\leq_{\mathrm{dZ}}$,   $\leq_{\operatorname{GD}}$ and $\leq_{\mathrm{GP}}$    are pre-linear orders satisfying the descending chain	condition.
\end{lem}

Combining all the orders above, we can now construct an order $\leq_{\operatorname{PD}}$ of $\frakS(Z)$:
$$u \leq_{\operatorname{PD}}v \Leftrightarrow \left\{
\begin{array}{lcl}
	u \leq_{\mathrm{D}} v ,\text{or }\\
	u =_{\mathrm{D}} v \text{ and } u \leq_{\mathrm{P}} v , \text{or }\\
	u =_{\mathrm{D}} v , u =_{\mathrm{P}} v \text{ and } u \leq_{\mathrm{dZ}} v, \text{or }\\
	 u =_{\mathrm{D}} v , u =_{\mathrm{P}} v , u =_{\mathrm{dZ}} v \text{ and } u \leq_{\mathrm{GD}} v , \text{or }\\
      u =_{\mathrm{D}} v , u =_{\mathrm{P}} v , u =_{\mathrm{dZ}} v, u =_{\mathrm{GD}} v \text{ and }  u \leq_{\mathrm{GP}} v , \text{or }\\
      u =_{\mathrm{D}} v , u =_{\mathrm{P}} v , u =_{\mathrm{dZ}} v, u =_{\mathrm{GD}} v ,  u =_{\mathrm{GP}} v \text{ and } u \leq_{\mathrm{Dlex}} v.
\end{array}
\right.	$$

To  prove that the $\leq_{\operatorname{PD}}$ is a well-order, we need some preparation.

\begin{defn} 	\begin{itemize}
		\item [(a)] Given some preorders $\leq_{\alpha_{1}}, \dots, \leq_{\alpha_{k}}$ on a set $Z$ with $k\geq 2$, introduce another preorder
$\leq_{\alpha_{1}, \dots, \alpha_{k}}$ by imposing recursively
$$
	u \leq_{\alpha_{1}, \dots, \alpha_{k}} v \Leftrightarrow\left\{
	\begin{array}{l}
		u<_{\alpha_{1}} v,  \text{ or }  \\
		 u=_{\alpha_{1}} v \text { and } u \leq_{\alpha_{2}, \dots, \alpha_{k}} v.
	\end{array}\right.
	$$
\item [(b)] 	Let $k \geq2$ and let $\leq_{\alpha_i}$ be a pre-linear order on $Z_{i},~1 \leq i \leq k$. Define the lexicographical product order $\leq_{\text{clex}}$ on the cartesian product $Z_{1} \times Z_{2} \times \cdots \times Z_{k}$ by defining
$$(x_{1},\cdots, x_{k}) \leq_{\text {clex}}(y_{1},\cdots, y_{k}) \Leftrightarrow\left\{\begin{array}{l}x_{1}<_{\alpha_{1}} y_{1}, \text {or } \\  x_{1}=_{Z_{1}}y_{1} \text { and }\left(x_{2}, \cdots, x_{k}\right) \leq_{\rm{clex}}\left(y_{2}, \cdots, y_{k}\right),\end{array}\right.$$ where $\left(x_{2}, \cdots, x_{k}\right) \leq_{\rm{clex}}\left(y_{2}, \cdots, y_{k}\right)$ is defined by induction, with the convention that $\leq_{\rm{clex}}$ is the trivial relation when $k=1$.
\end{itemize}
\end{defn}

\begin{lem}[{\cite[Lemma 1.7]{QQZ21}}] \label{sequence of order gives linear order}
	\begin{itemize}
		\item[(a)]	For $k \geq  2$, let $\leq_{\alpha_{1}}, \dots, \leq_{\alpha_{k-1}}$ be pre-linear orders on $Z$, and $\leq_{\alpha_{k}}$ a linear order on $Z$. Then   $\leq_{\alpha_{1}, \dots, \alpha_{k}}$ is a linear order on $ Z$.
		\item[(b)] Let $ \leq_{\alpha_{i}} $ be a well order on $Z_i$, $1\leq i\leq k$.  Then the lexicographical product
		order $\leq_{\rm{clex}}$ is a well order on the cartesian product $Z_1\times Z_2\times \cdots \times Z_k$.
	
	\end{itemize}
\end{lem}

\begin{prop}
	The order   $\leq_{\operatorname{PD}}$ is a well order on $\frakS(Z)$.
\end{prop}

\begin{proof}Since $\leq_{\mathrm{Dlex}}$ is a linear order, so is $\leq_{\operatorname{PD}}$ by Lemma~\ref{lemma: prelinear order with descending chain	condition} and Lemma~\ref{sequence of order gives linear order}(a).

It suffices to	verify that $\leq_{\operatorname{PD}}$ satisfies the descending chain condition. Let $$v_1\geq_{\operatorname{PD}} v_2\geq_{\operatorname{PD}} v_3\geq_{\operatorname{PD}}\cdots \in\frakS(Z) $$ be a descending chain. By Lemma~\ref{lemma: prelinear order with descending chain	condition},  there exist $N\geq 1$  such that $$\deg_D(v_N)=\deg_D(v_{N+1})=\deg_D(v_{N+2})=\cdots=:k,$$
$$\deg_P(v_N)=\deg_P(v_{N+1})=\deg_P(v_{N+2})=\cdots=:p,$$ $$\deg_Z(v_N)=\deg_Z(v_{N+1})=\deg_Z(v_{N+2}) =\cdots$$  $$\deg_{GD}(v_N)=\deg_{GD}(v_{N+1})=\deg_{GD}(v_{N+2})=\cdots,$$ and $$\deg_{GP}(v_N)=\deg_{GP}(v_{N+1})=\deg_{GP}(v_{N+2})=\cdots.$$
	Thus all $v_i$ with $i\geq N$  belong to $\frakS_{,k+p}(Z)$. The restriction of the order $\leq_{\mathrm{Dlex}}$ to $\frakS_{,k+p}(Z)$ equals
	to the well order $\leq_{\operatorname{Dlex}_{k+p}}$, which by definition satisfies the descending chain condition, so the chain $v_1\geq_{\operatorname{PD}} v_2\geq_{\operatorname{PD}} v_3\geq_{\operatorname{PD}}\cdots$ stabilizes after finite steps.
\end{proof}

\begin{defn}[{\cite[Definition~5.6]{GGSZ14}}]
	A preorder   $\leq_{\alpha}$ on $\frakS(Z)$ is called bracket compatible (resp. left compatible, right compatible) if
	$$u \leq_{\alpha} v \Rightarrow\lfloor u\rfloor_D \leq_{\alpha}\lfloor v\rfloor_D \text{ and } \lfloor u\rfloor_P \leq_{\alpha}\lfloor v\rfloor_P , \text { (resp. } w u \leq_{\alpha} w v, ~u w \leq_{\alpha} v w ,\  \text { for all }  w \in \frakS(Z))\
 $$
\end{defn}

\begin{lem}[{\cite[Lemma~5.7]{GGSZ14}}]\label{monomial order lemma}
	A well order $\leq$ is a monomial order on $\frakS(Z)$  if and only if $ \leq$ is bracket compatible, left compatible and right compatible.
\end{lem}

Now we can prove the main result of this section which is the main technical point of this paper.
\begin{thm}
	The well order $\leq_{\operatorname{PD}}$  is a monomial order on $\frakS(Z)$.
\end{thm}

\begin{proof}

Let $u\leq_{\operatorname{PD}}v$.
	It is obvious that preorders $\leq_{\mathrm{D}}$, $\leq_{\mathrm{P}}$ and $\leq_{\mathrm{dZ}}$ are bracket compatible, left compatible and right compatible. This solves the three cases  $u<_{\mathrm{D}} v$; $u=_{\mathrm{D}} v$,  $u<_{\dgp} v$;     $u=_{\mathrm{D}} v$, $u=_{\dgp} v$ and $u<_{\dgx} v$.

If $u=_{\mathrm{D}} v$, $u=_{\mathrm{P}} v, u=_{\mathrm{dZ}} v$  and   $u <_{\mathrm{GD}}v$,
  obviously    $\lfloor u\rfloor_D <_{\mathrm{GD}}\lfloor v\rfloor_D$, $\lfloor u\rfloor_P <_{\mathrm{GD}}\lfloor v\rfloor_P$ $uw <_{\mathrm{GD}}vw$ and   $wu <_{\mathrm{GD}}wv$ for $w\in \frakS(Z)$.  So $\lfloor u\rfloor_D <_{\mathrm{PD}}\lfloor v\rfloor_D$, $\lfloor u\rfloor_P <_{\mathrm{PD}}\lfloor v\rfloor_P$, $uw <_{\mathrm{PD}}vw$ and   $wu <_{\mathrm{PD}}wv$.

  The case that  $u=_{\mathrm{D}} v$, $u=_{\mathrm{P}} v, u=_{\mathrm{dZ}} v$, $u =_{\mathrm{GD}}v$ and 	$u <_{\mathrm{GP}}v$ is similar to the above one.
	
	It remains to  consider the case that  $u=_{\mathrm{D}} v$, $u=_{\mathrm{P}} v, u=_{\mathrm{dZ}} v$, $u =_{\mathrm{GD}}v$, $u =_{\mathrm{GP}}v$ and 	$u <_{\mathrm{Dlex}}v$.
	Let $n\geq \deg_D(u), \deg_P(u)$. Since $u,v\in \frakS_{,n}(Z)$, thus  $u\leq_{\operatorname{Dlex}_{n}} v$. By the fact that  the restriction of $\leq_{\operatorname{Dlex}_{n+1}} $ to $\lfloor\frakS_{,n}(Z)\rfloor_D$ is induced by  $ \leq_{\operatorname{Dlex}_n}$, we have $\lfloor u\rfloor_D \leq_{\operatorname{Dlex}_{n+1}} \lfloor v\rfloor_D$, $\lfloor u\rfloor_D \leq_{\operatorname{Dlex}} \lfloor v\rfloor_D$,  and $\lfloor u\rfloor_D \leq_{\operatorname{PD}} \lfloor v\rfloor_D$. Similarly $\lfloor u\rfloor_P \leq_{\operatorname{PD}} \lfloor v\rfloor_P$.
Let $w\in\frakS_{,m}(Z)$.  One can obtain   $uw\leq_{\operatorname{Dlex}_r}vw$ and $wu\leq_{\operatorname{Dlex}_r}wv$ for $r=\max \left\lbrace m, n \right\rbrace $, so $uw\leq_{\operatorname{PD}} vw$ and $wu\leq_{\operatorname{PD}}wv$.

We are done.
\end{proof}

Now let's move   to the unital case.
Now we extend $\leq_{\operatorname{PD}} $ from $\frakS(Z)$ to $\frakM(Z)$ by using
Remark~\ref{remark: monoids}.
\begin{defn}\label{udl}
	Let $Z$ be a set with a well order. Let  $\dagger_P$ (resp. $\dagger_D$ ) be a symbol which is understood to be  $\lfloor 1 \rfloor_P$ (resp.  $\lfloor 1 \rfloor_D$) and write $Z'=Z\sqcup \{\dagger_P ,\dagger_D\}$.  Consider the free operated semigroup $\frakS(Z')$  over the set  $Z'$. The well order on $Z$ extends to a  well order $\leq$ on $Z'$ by setting $\dagger_P>z>\dagger_D$, for any $z\in Z$. Besides, we impose  $\operatorname{deg}_{P}(\dagger_P) = 1$ and $\deg_{GP}(\dagger_P)=0$.
	Then the monomial order $\leq_{\operatorname{PD}} $ on $\frakS(Z')$ induces a well order $\leq_{\operatorname{uPD}}$ on $\frakM(Z)=\frakS(Z')\sqcup \{1\}$ (in which $\lfloor 1 \rfloor_P$ and $\lfloor 1 \rfloor_D$ is identified with $\dagger_P$ and $\dagger_D$ respectively), by setting $u>_{\operatorname{uPD}}1$ for any $u\in \frakS(Z')$.
\end{defn}


\begin{thm}
	The well order $\leq_{\operatorname{uPD}}$  is a monomial order on $\frakM(Z)$.
\end{thm}

\begin{proof} Obviously,
the well order $\leq_{\operatorname{uPD}}$ is bracket compatible on $\frakM(Z)\backslash \{1\}$.
Let  $x\in \frakM(Z)\backslash \{1\}$. By definition, $x>_{\operatorname{uPD}}1$. We have  $\lfloor x\rfloor_P>_{\mathrm{Dlex}}\lfloor 1\rfloor_P$ which implies    $\lfloor x\rfloor_P>_{\operatorname{uPD}}\dagger_P$.
  It is ready to see that $\lfloor x\rfloor_D>_{\operatorname{uPD}} x >_{\operatorname{uPD}}\dagger_D$.
Thus $\leq_{\operatorname{uPD}}$ is bracket compatible.

 Clearly,  $\leq_{\operatorname{uPD}}$ is left and right compatible.
\end{proof}

We  record  several important conclusions which will be useful later.

\begin{prop}\label{uDl}
For any $ u,v\in \frakM(Z)\backslash \{1\}$, we have

 	\begin{itemize}	
		\item [(a)] $\lfloor u\rfloor_P\lfloor 1\rfloor_P >_{\operatorname{uPD}} \lfloor u\lfloor 1\rfloor_P \rfloor_P \geq_{\operatorname{uPD}} \lfloor \lfloor u\rfloor_P \rfloor_P $,

		\item [(b)] $\lfloor 1\rfloor_P\lfloor v \rfloor_P>_{\operatorname{uPD}} \lfloor \lfloor v\rfloor_P \rfloor_P  \geq_{\operatorname{uPD}} \lfloor \lfloor 1\rfloor_P v \rfloor_P  $,

		\item [(c)] $\lfloor 1\rfloor_P\lfloor 1\rfloor_P>_{\operatorname{uPD}} \lfloor \lfloor 1\rfloor_P\rfloor_P$,

		\item [(d)] $\lfloor 1\rfloor_P\lfloor v\rfloor_D>_{\operatorname{uPD}} \lfloor \lfloor 1\rfloor_P v\rfloor_D$,

		\item [(e)] $\lfloor u\rfloor_D\lfloor 1\rfloor_P>_{\operatorname{uPD}} \lfloor u\lfloor 1\rfloor_P\rfloor_D$.
 	\end{itemize}
\end{prop}

\begin{proof}
Let  $u , v\in  \frakM(Z)\backslash \{1\}=\frakS(Z')$.
	 	\begin{itemize}	

		\item [(a)]
It is easy to see $\lfloor \lfloor u\rfloor_P \rfloor_P$ have lowest $\deg_{Z'}$ among $\lfloor u\rfloor_P\lfloor 1\rfloor_P, \lfloor u\lfloor 1\rfloor_P \rfloor_P , \lfloor \lfloor u\rfloor_P \rfloor_P $, and
 we also have $\deg_{GP}(\lfloor u\rfloor_P\lfloor 1\rfloor_P)>\deg_{GP}(\lfloor u\lfloor 1\rfloor_P \rfloor_P)$.

		\item [(b)] It is similar to $(a)$.

		\item [(c)] It follows from  $\deg_{Z'}(\lfloor 1\rfloor_P\lfloor 1\rfloor_P)>\deg_{Z'}(\lfloor \lfloor 1\rfloor_P\rfloor_P)$.

		\item [(d)] It can be deduced from  $\lfloor 1\rfloor_P\lfloor v\rfloor_D >_{\operatorname{Dlex}} \lfloor \lfloor 1\rfloor_P v\rfloor_D$ by  Definition~\ref{Def: deg-lex order}.

		\item [(e)]It holds because  $\deg_{GD}(\lfloor u\rfloor_D\lfloor 1\rfloor_P)>\deg_{GD}(\lfloor u\lfloor 1\rfloor_P\rfloor_D)$.
	 	\end{itemize}
\end{proof}

\section{Operator  polynomial identities and multi-operated  GS bases}\label{Section: multi-operated  GS bases}

In this section, we extend  the theory of   operated  GS bases due to  \cite{BokutChenQiu, GSZ13, GGSZ14, GaoGuo17}  from the case of single operator to    multiple operators case. The presentation is essentially contained in \cite{GGR15}.


\subsection{Operator  polynomial identities}\


In this subsection, we   give   some basic notions and facts   related to   operator  polynomial identities.
Throughout this section, $X$ denotes a set.

\begin{defn}
	   We call an element     $\phi(x_1,\dots,x_n)   \in \bk\frakS(X) $ (resp. $\bk\frakM (X)$) with $ n\geq 1, x_1, \dots, x_n\in X$ an operated polynomial identity (aka  OPI).

\end{defn}
\textit{{From now on, we always assume that OPIs are multilinear, that is,  they are linear in each $x_i$.}}

\begin{defn}  Let $\phi(x_1,\dots,x_n)$ be an   OPI. A  (unital)  $\Omega$-algebra $A=(A,\{P_\omega\}_{\omega\in \Omega})$ is said to satisfy the OPI $\phi(x_1,\dots,x_n)$ if
	$\phi(r_1,\dots,r_n)=0,$ for all $r_1,\dots,r_n\in A.$
In this case, $(A,\{P_\omega\}_{\omega\in \Omega})$ is called a (unital) $\phi$-algebra.  
	
	Generally, for a family $\Phi$ of OPIs, we call a  (unital)   $\Omega$-algebra    $(A, \{P_\omega\}_{\omega\in \Omega})$ a  (unital)  $\Phi$-algebra    if it is a  (unital) $\phi$-algebra for any $\phi\in \Phi$.
	Denote the  category of $\Phi$-algebras (resp. unital $\Phi$-algebras)  by $\Phi\zhx\Alg$ (resp. $\Phi\zhx\uAlg$).

\end{defn}

\begin{defn}   An $\Omega$-ideal  of an $\Omega$-algebra  is   an  ideal   of the associative algebra  closed under the action of the operators.
	The $\Omega$-ideal generated by a subset $S\subseteq A$ is denoted by $\left\langle S\right\rangle_{\mtOpAlg}$ (resp. $\left\langle S\right\rangle_{\mtuOpAlg}$).
\end{defn}
Obviously   the quotient of an $\Omega$-algebra (resp. unital $\Omega$-algebra) by an $\Omega$-ideal is naturally an $\Omega$-algebra (resp.  $\Omega$-unital algebra).

	From now on, $\Phi$ denotes  a family  of OPIs    in $\bk\frakS(X) $ or   $\bk\frakM (X)$.  For a set $Z$ and  a subset  $Y$    of $ \frakM (Z)$, introduce the subset $S_\Phi(Y)\subseteq\bk \frakM (Z)$ to be
\[S_{\Phi}(Y):=\{\phi(u_1,\dots,u_k)\ |\ u_1,\dots,u_k\in Y,~\phi(x_1,\dots,x_k)\in \Phi\}.\]

\subsection{Multi-operated GS  bases for   $\Phi$-algebras}\

In this subsection, operated GS basis theory is extended to algebras with  multiple operators  following  closely \cite{BokutChenQiu}.

\begin{defn}
	Let $Z$ be a set,  $\leq$ a linear order on $\frakM(Z)$ and $f \in \bk \frakM(Z)$.
	\begin{itemize}
		\item [(a)] Let $f\notin \bk$. The leading monomial of $f$ , denoted by $\bar{f}$, is the largest monomial appearing in $f$. The leading coefficient of $f$ , denoted by $c_f$, is the coefficient of $\bar{f}$ in $f$. We call $f$	monic with respect to $\leq$ if $c_f = 1$.
		\item [(a')] Let $f\in \bk$ (including the case $f=0$).  We define the leading monomial of $f$ to be $1$ and the	leading coefficient of $f$ to be $c_f=f$.
		\item [(b)] A subset $S\subseteq \bk\frakM(Z)$ is called monicized with respect to $\leq$,  if each nonzero element of $S$ has leading coefficient $1$.
	\end{itemize}
\end{defn}
Obviously, each subset $S\subseteq \frakM(Z)$ can be made monicized  if we divide  each nonzero  element by    its  leading coefficient.

We need  another notation.
Let $Z$ be a set. For  $u\in\frakM(Z)$ with $u\neq1$, as  $u$ can be uniquely written as a product $ u_1 \cdots u_n $ with $ u_i \in Z \cup \left \lfloor\frakM(Z)\right \rfloor_\Omega$ for $1\leq i\leq n$, call $n$ the breadth of $u$, denoted by $|u|$; for $u=1$, we define
$ |u| = 0 $.

\begin{defn}
	Let $\leq$ be a monomial order on $\frakS(Z)$ (resp. $\frakM(Z)$)  and $ f, g \in \bk\frakS(Z)$ (resp. $\bk\frakM(Z)$) be monic.
	\begin{itemize}
		\item[(a)] If there are $p,u,v\in \frakS(Z)$ (resp. $\frakM(Z)$) such that $p=\bar{f}u=v\bar{g}$ with max$\left\lbrace |\bar{f}| ,|\bar{g}|   \right\rbrace < \left| p\right| < |\bar{f}| +|\bar{g}|$, we call
		$$\left( f,g\right) ^{u,v}_p:=fu-vg$$
		the intersection composition of $f$ and $g$ with respect to $p$.
		\item[(b)] If there are $p\in\frakS(Z)$ (resp. $\frakM(Z)$) and $q\in\frakSstar(Z)$ (resp. $\frakMstar(Z)$) such that $p=\bar{f}=q|_{\bar{g}}$, we call
		$$\left( f,g\right)^q_p:=f-q|_g $$
		the inclusion composition of $f$ and $g$ with respect to $p$.
	\end{itemize}
\end{defn}
\begin{defn}
	Let $Z$ be a set and $\leq$ a monomial order on $\frakS(Z)$ (resp. $\frakM(Z)$). Let $\calg\subseteq \bk\frakS(Z) $ (resp. $\bk\frakM(Z)$).
	\begin{itemize}
		\item[(a)]An element $f\in \bk\frakS(Z) $ (resp. $\bk\frakM(Z) $) is called trivial modulo $(\calg,p)$ for $p\in \frakS(Z)$ (resp. $\frakM(Z)$) if
		$$f=\underset{i}{\sum}c_iq_i|_{s_i}\text{ with }q_i|_{\bar s_i}<p\text{, where }c_i\in \bk,\ q_i\in\frakSstar(Z)~(\text{resp.}~\frakMstar(Z))\  \mathrm{and}\  s_i\in \calg.$$
If this is the case, we write
$$
(f, g)_p \equiv 0 \bmod (\calg, p) .
$$
In general, for any  $u, v\in \frakS(Z)$ (resp. $\frakM(Z)$), $u \equiv v \bmod (\calg, p)$ means that $u-v=\left.\sum c_i q_i\right|_{s_i}$, with $q_i|_{\bar s_i}<p$, where $c_i\in \bk,\ q_i\in\frakSstar(Z)$ (resp. $\frakMstar(Z))$ and $s_i\in \calg$.
		\item[(b)]  The subset $\calg$ is called a  GS basis in $\bk\frakS(Z)$ (resp. $\bk\frakM(Z)$) with respect to $\leq$ if, for all pairs $f,g\in \calg$ monicized with respect to $\leq$, every  intersection composition of the form $\left( f,g\right) ^{u,v}_p$ is trivial modulo $(\calg,p)$,  and every inclusion composition of the form $\left( f,g\right)^q_p$ is trivial modulo $(\calg,p)$.
	\end{itemize}
\end{defn}

\textit{ {To distinguish from usual GS bases for associative algebras,  from now on, we shall rename       GS bases  in  multi-operated contexts by $\Omega$-GS bases.}}

\begin{thm}  \label{Thm: unital CD}
	(Composition-Diamond Lemma) Let $Z$ be a set, $\leq$ a monomial order on $\frakM(Z)$ and $\calg\subseteq \bk\frakM(Z) $. Then the following conditions are equivalent:
	\begin{itemize}
		\item[(a)] $\calg$ is an $\Omega$-GS basis in $\bk\frakM(Z) $.
		\item[(b)]  Denote
		$$\Irr(\calg):=\frakM(Z)\backslash \left\lbrace q|_{\overline{s}}~|~s\in \calg,\ \ q\in\frakMstar(Z)\right\rbrace. $$
		As a $\bk$-space, $\bk\frakM(Z) =\bk\Irr(\calg)\oplus\left\langle \calg \right\rangle_\mtuOpAlg$ and     $ \Irr(\calg) $ is a $\bk$-basis of $\bk\frakM(Z)\slash\left\langle\calg \right\rangle_\mtuOpAlg$.
	\end{itemize}
\end{thm}

\begin{thm} \label{Thm: nonunital CD}
	(Composition-Diamond Lemma) Let $Z$ be a set, $\leq$ a monomial order on $\frakS(Z)$ and $\calg\subseteq \bk\frakS(Z) $. Then the following conditions are equivalent:
	\begin{itemize}
		\item[(a)] $\calg$ is an $\Omega$-GS basis in $\bk\frakS(Z) $.
		\item[(b)]   Denote
		$$\Irr(\calg):=\frakS(Z)\backslash \left\lbrace q|_{\overline{s}}~|~s\in \calg,\ \ q\in\frakSstar(Z)\right\rbrace. $$
		As a $\bk$-space, $\bk\frakS(Z) =\bk\Irr(\calg)\oplus\left\langle \calg \right\rangle_\mtOpAlg$ and     $ \Irr(\calg)$ is a $\bk$-basis of $\bk\frakS(Z)\slash\left\langle\calg \right\rangle_\mtOpAlg$.
	\end{itemize}
\end{thm}

\begin{defn}[{\cite[Definiton~2.21(a)]{GaoGuo17}}]\label{def: Omega-GS}\begin{itemize}
		\item [(a)]	Let  $\Phi\subseteq \bk\frakS(X)$ be  a  family   of OPIs. Let $Z$ be a set and $\leq$ a monomial
		order on $\frakS(Z)$. We call $\Phi$  $\Omega$-GS  on $\bk\frakS(Z)$  with respect to $\leq$ if $S_{\Phi}(\frakS(Z))$    is an  $\Omega$-GS  basis  in $\bk\frakS(Z)$   with respect to $\leq$.
		\item [(b)]Let   $\Phi\subseteq  \bk\frakM(X)$ be  a  family  of OPIs. Let $Z$ be a set and $\leq$ a monomial
		order on  $\frakM(Z)$. We call $\Phi$   $\Omega$-GS  on  $\bk\frakM(Z)$  with respect to $\leq$ if  $S_{\Phi}(\frakM(Z))$    is an $\Omega$-GS basis  in   $\bk\frakM(Z)$  with respect to $\leq$.
	\end{itemize}

\end{defn}


\subsection{Multi-operated GS basis for free  $\Phi$-algebras over   algebras}\

In this subsection, we consider multi-operated GS basis for free  $\Phi$-algebras over   algebras and  generalise the main result of \cite{QQWZ21} to   multi-operated cases.

%
%
%
We will use the following results without proof as they are counterparts in multi-operated setup of \cite[Propositions 4.8]{QQWZ21}.

\begin{prop}\label{free phi algebra}
	\begin{itemize}
		\item [(a)] Let $\Phi\subset\bk\frakS(X)$ and $A=\bk \cals (Z) \slash I_A$ an algebra. Then
		$$ \calf_{\Alg}^{\Phi\zhx\Alg}(A):=\bk\frakS(Z)\slash\left\langle  S_{\Phi}(\frakS(Z))\cup I_A\right\rangle_\mtOpAlg$$
		is the free $\Phi$-algebra generated by  $A$.
		\item [(b)] Let $\Phi\subset\bk\frakM(X)$ and $A=\bk \calm (Z) \slash I_A$ a unital algebra. Then
		 $$\calf_{\uAlg}^{\Phi\zhx\uAlg}(A):=\bk\frakM(Z)\slash\left\langle  S_{\Phi}(\frakM(Z))\cup I_A\right\rangle_\mtuOpAlg$$
		 is the free unital $\Phi$-algebra over $A$.
	\end{itemize}
\end{prop}

As in \cite{QQWZ21}, we consider the following question:
\begin{Ques}  Let $A$ be  a (unital) algebra  together   with a   Gr\"obner-Shirshov basis $G$. Assume that a set $\Phi$ of operated  polynomial identities is   $\Omega$-GS in the sense of  Definition~\ref{def: Omega-GS}.
Considering the free (unital)  $\Phi$-algebra $B$ over $A$,   when will
  the union ``$\Phi\cup G$" be  a  $\Omega$-GS  basis  for $B$?
\end{Ques}

It is surprising that the answer of the corresponding question given in
\cite{QQWZ21} can be generalised to multi-operated case without much modifications.


\begin{thm} \label{Thm: GS basis for free unital Phi algebra over unital alg}
	Let $X$ be a set and $\Phi\subseteq \bk\frakM (X)$ a system of OPIs. Let $A=\bk \calm (Z) \slash I_A$ be a unital algebra with generating set $Z$.
	Assume that $\Phi$ is $\Omega$-GS  on $Z$ with respect to a monomial order $\leq$ in $\frakM (Z)$ and that  $G$ is a GS    basis of $I_{A}$ in $\bk \calm (Z)$ with respect to the restriction of $\leq$ to $ \calm(Z)$.

	Suppose that the leading monomial  of any OPI $\phi(x_1, \dots, x_n)\in \Phi$ has no subword in $\calm(X)\backslash X$,  and  that   $\phi(u_1, \dots, u_n) $   vanishes or its  leading monomial is      $\overline{\phi}(u_1, \dots, u_n) $  for all $u_1, \dots, u_n\in \frakM  (Z)$.  Then $S_{\Phi}(\frakM (Z))\cup G$ is an $\Omega$-GS basis of $\left\langle S_{\Phi}(\frakM (Z))\cup I_A\right\rangle_{\mtuOpAlg}$ in $\bk\frakM (Z)$ with respect to $\leq$.
\end{thm}

\begin{proof}

The proof of  \cite[Theorem~5.9]{QQWZ21} carries verbatim over multi-operated case, because it reveals that   the key point   is that the leading monomial  of any OPI $\phi(x_1, \dots, x_n)\in \Phi$ has no subword in $\calm(X)\backslash X$.

For details, see the proof of \cite[Theorem~5.9]{QQWZ21}.
\end{proof}

%

There exists a nonunital version of the above  result, which is also a multi-operated version of \cite[Theorem 2.15]{QQZ21}. 

\begin{thm}\label{Thm: GS basis for free nonunital Phi algebra over nonunital alg}
	Let $X$ be a set and $\Phi\subseteq \bk\frakS(X)$ a system of OPIs. Let $A=\bk \cals (Z) \slash I_A$ be an algebra with generating set $Z$.
	Assume that $\Phi$ is $\Omega$-GS  on $Z$ with respect to a monomial order $\leq$ in $\frakS(Z)$ and that  $G$ is a GS    basis of $I_{A}$ in $\bk \cals(Z)$ with respect to the restriction of $\leq$ to $ \cals(Z)$.
	
	Suppose that the leading monomial  of any OPI $\phi(x_1, \dots, x_n)\in \Phi$ has no subword in $\cals(X)\backslash X$,  and  that  for all $u_1, \dots, u_n\in \frakS (Z)$,  $\phi(u_1, \dots, u_n) $   vanishes or its  leading monomial is      $\overline{\phi}(u_1, \dots, u_n) $.  Then $S_{\Phi}(\frakS(Z))\cup G$ is an $\Omega$-GS basis of $\left\langle S_{\Phi}(\frakS(Z))\cup I_A\right\rangle_\mtOpAlg$ in $\bk\frakS(Z)$ with respect to $\leq$.

\end{thm}

\section{Free differential Rota-Baxter algebras over algebras }\label{Section: differential Rota-Baxter algebras}

In this section, we apply Theorems~\ref{Thm: GS basis for free unital Phi algebra over unital alg} and \ref{Thm: GS basis for free nonunital Phi algebra over nonunital alg} to differential Rota-Baxter algebras.

From now on, let  $\Omega=\{D, P\}$, fix a set $X=\{x, y\}$ with two elements such that variables in OPIs will take values in $X$.
When talking about algebras or reductions of OPIs, fix a set $Z$ and we understand that variables in OPIs will be replaced by elements of
$\frakS(Z)$ or $\frakM(Z)$.

We first recall  the definition  of differential Rota-Baxter algebras. We use $D(~)$ and $P(~)$ instead of the linear operators $\lfloor ~ \rfloor_D$ and $\lfloor ~ \rfloor_P$.

\begin{defn}[{\cite[Definition~2.1]{GGR15}}]\label{differentila rota-baxter algebras}
 Let $\lambda \in \bk$ be fixed.
\begin{itemize}
		\item[(a)] A (unital) differential $\bk$-algebra of weight $\lambda$ (also called a (unital) $\lambda$-differential $\bk$-algebra) is a (unital) associative $\bk$-algebra $R$ together with a linear operator $D: R \rightarrow R$ such that
$$
 \quad D(u v)=D(u) v+u D(v)+\lambda D(u) D(v) \quad \text { for all } u, v \in R;
$$
when $R$ has a unity $1$,  it is asked that $D(1)=0$.

\item[(b)] A Rota-Baxter $\bk$-algebra of weight $\lambda$ is an associative $\bk$-algebra $R$ together with a linear operator $P: R \rightarrow R$ such that
$$
P(u) P(v)=P(u P(v))+P(P(u) v)+\lambda P(u v) \quad \text { for all } u, v \in R.
$$
\item[(c)] A (unital) differential Rota-Baxter $\bk$-algebra of weight $\lambda$ (also called a  (unital) $\lambda$-differential Rota-Baxter $\bk$-algebra) is a  (unital) differential k-algebra $(R, D)$ of weight $\lambda$ and a Rota-Baxter operator $P$ of weight $\lambda$ such that
$$
D \circ P=\text { id }.
$$

\end{itemize}
\end{defn}

When we consider free  differential Rota-Baxter algebras on algebras, it is disappointing  to see that the traditional order (see \cite{BokutChenQiu})  would  not meet the condition of Theorems~\ref{Thm: GS basis for free unital Phi algebra over unital alg} and \ref{Thm: GS basis for free nonunital Phi algebra over nonunital alg}.  This is the intention  of the new monomial orders $\leq_{\operatorname{PD}}$ and $\leq_{\operatorname{uPD}}$ introduced in Section~\ref{section: order}.

\medskip

\subsection{ Case of nonunital algebras with $\lambda\neq 0$} \label{Subsection: Case of lambda non zero}\


Assume in this subsection that $\lambda\neq 0$. Denote
\begin{itemize}
     \item[(1)]  $\phi_1(x,y) = P(x) P(y)-P(x P(y))-P(P(x) y)-\lambda P(x y)$,
   \item[(2)]  $\phi_2(x,y) = D(x)D(y) + \lambda^{-1}D(x)y + \lambda^{-1}xD(y) - \lambda^{-1}D(xy)$,
    \item[(3)]  $\phi_3(x) = D(P(x))-x$.
\end{itemize}

We first consider   nonunital free differential Rota-Baxter algebras  on algebras.


\begin{prop}For any $u, v \in \frakS(Z)$,  the leading monomials of $\phi_1(u, v)$,  $\phi_2(u, v)$ and $\phi_3(u)$ under $\leq_{\operatorname{PD}}$ are respectively
 $ P(u)P(v),  D(u)D(v)$ and $D(P(u))$.

\end{prop}

\begin{proof}
Let $u_1, \cdots, u_n$ and $v_1, \cdots, v_m$ be all the elements of $Z$ occurring in $u$ and $v$ from left to right.


For $\phi_1(u,v) = P(u) P(v)-P(u P(v))-P(P(u) v)-\lambda P(u v)$, we have $\deg_P(P(uv))$ is smaller than those of the other three  terms, while  the $\deg_D, \deg_Z$ and $\deg_{GD}$ of the other  elements are the same. And one can see  $$\deg_{GP}(P(u)P(v))-\deg_{GP}(P(uP(v)))=m>0,$$  $$\deg_{GP}(P(u)P(v))-\deg_{GP}(P(P(u)v))=n>0,$$ so   the leading monomial of $\phi_1(u, v)$ is $P(u)P(v)$.

The statements about  $\phi_2(u, v)$ and $\phi_3(u)$ are obvious by comparing   $\deg_D$.
\end{proof}

Now let   $$\dRB':=\left\lbrace\phi_1(x, y) , \phi_2(x,y) , \phi_3(x) \right\rbrace. $$
However,  $\dRB'$  is not  $\Omega$-GS in $\bk\frakS(Z)$ with respect to $\leq_{\operatorname{PD}}$.
\begin{exam}\label{phi_4}
For $u, v\in \frakS(Z)$, let
 $$\begin{array}{rcl}  f&=& \phi_2(P(u),v)=D(P(u))D(v) + \lambda^{-1}D(P(u))v + \lambda^{-1}P(u)D(v) - \lambda^{-1}D(P(u)v),\\\
   g &=& \phi_3(u)=D(P(u))-u,\\
   q&=&\star D(v),\\
   p&=&D(P(u))D(v)=\bar{f}=\left.q\right|_{\bar{g}}. \end{array}$$
Then $$
(f,g)_p^q=f-\left.q\right|_{g}\equiv \lambda^{-1}( P(u)D(v)-D(P(u)v)+uv+\lambda uD(v)).
$$
Let $$\phi_4(x,y)=P(x)D(y)-D(P(x)y)+xy+\lambda xD(y).$$  It is clear  that the leading monomial  of $\phi_4(u,v)$ is   $P(u)D(v)$ with respect to $\leq_{\operatorname{PD}} $ which cannot be reduced further.
\end{exam}
\begin{exam}\label{phi_5}
For $u, v\in \frakS(Z)$, let
 $$\begin{array}{rcl}   f&=& \phi_2(u,P(v))=D(u)D(P(v)) + \lambda^{-1}D(u)P(v) + \lambda^{-1}uD(P(v)) - \lambda^{-1}D(uP(v)),\\
   g &=& \phi_3(v)=D(P(v))-v,\\
   q&=&D(u)\star, \\
     p&=&D(u)D(P(v))=\bar{f}=\left.q\right|_{\bar{g}}.\end{array}$$
 Then
$$
(f,g)_p^q=f-\left.q\right|_{g}\equiv \lambda^{-1}(D(u)P(v)-D(uP(v))+uv+\lambda D(u)v).
$$
Let  $$\phi_5(x,y)=D(x)P(y)-D(xP(y))+xy+\lambda D(x)y.$$
It is clear  that the leading monomial  of $ \phi_5(u,v) $ is    $D(u)P(v)$ with respect to $\leq_{\operatorname{PD}} $ which cannot  be reduced further.
\end{exam}
Now denote  $\dRB$ to be the set of the following OPIs:
\begin{itemize}
		\item [(1)] $\phi_1(x,y) =P(x)P(y) - P(xP(y)) - P(P(x)y) -  \lambda P(xy)$,

		\item [(2)] $ \phi_2(x,y)=D(x)D(y) + \lambda^{-1}D(x)y + \lambda^{-1}xD(y) - \lambda^{-1}D(xy) $,

		\item [(3)] $\phi_3(x)=D(P(x)) - x$,

		\item [(4)] $\phi_4(x,y)=P(x)D(y)-D(P(x)y)+xy+\lambda xD(y)$,

		\item [(5)] $\phi_5(x,y)=D(x)P(y)-D(xP(y))+xy+\lambda D(x)y$.

 	\end{itemize}
 It is obvious that
$\left\langle S_{\dRB'}(Z)\right\rangle_\mtOpAlg=\left\langle S_{\dRB}(Z)\right\rangle_\mtOpAlg$ for each set $Z$.

Next we will show that  $\dRB$ is $\Omega$-GS with respect to $\leq_{\operatorname{PD}}$. Before that, we need the following lemma to simplify our proof.


\begin{lem}\label{including}
Let $\phi(x_1,\dots,x_n)$ and $\psi(y_1,\dots,y_m)$ be two OPIs.  Let $Z$ be a set. Suppose that, for any $u_1, \dots, u_n, v_1, \dots, v_m\in \frakS(Z)$, the leading monomial of $\phi(u_1,\dots,u_n)$ is $\bar\phi(u_1,\dots,u_n)$ and leading monomial of $\psi(v_1,\dots,v_m)$ is $\bar\psi(v_1,\dots,v_m)$.

Now write $f=\phi(u_1,\dots,u_n)$ and $g=\psi(v_1,\dots,v_m)$ for fixed $u_1, \dots, u_n$, $v_1, \dots, v_m\in \frakS(Z)$.
If there exists $i~(1 \leq i \leq n)$ and $r \in \frakS^{\star}(Z)$ such that
$u_i = \left.r\right|_{\bar{g}}$, then the inclusion composition $(f, g)_p^q = f - q|_g$ with $p=\bar f$ and $q = \bar{\phi}(u_1, \dots, u_{i-1}, r , u_{i+1}, \dots, u_n)$, is trivial modulo $(S_{\{\phi,\psi\}}(Z),w)$.  We call this type of inclusion composition as complete inclusion composition.
\end{lem}

\begin{proof}
The assertion follows from
$$
\begin{aligned}
	(f,g)_{p}^q
	&=f-q|_{g}\\
	&=(\phi-\bar\phi)(u_1, \dots, u_{i-1}, r|_{\bar g}, u_{i+1}, \dots,u_n)-\bar\phi(u_1,\dots, u_{i-1}, r|_{g-\bar g}, u_{i+1}, \dots,u_n)  \\
	&=(\phi-\bar\phi)(u_1, \dots, u_{i-1}, r|_{g}, u_{i+1}, \dots,u_n)-\phi(u_1,\dots, u_{i-1}, r|_{g-\bar g}, u_{i+1}, \dots,u_n).
\end{aligned}
$$
\end{proof}
%
\begin{remark}
  Lemma~\ref{including}  extends \cite[Theorem~4.1(b)]{GaoGuo17}  to the case of multiple operators.
\end{remark}

Now we can prove  $\dRB$ is  $\Omega$-GS.

\begin{thm}\label{S2}
 $\dRB$   is   $\Omega$-GS in $\bk\frakS(Z)$ with respect to $\leq_{\operatorname{PD}} $.
\end{thm}

\begin{proof}
We write  $i \wedge j$ the composition of OPIs of $\phi_i$ and $\phi_j$, which means $\phi_i$ lies on the left and $\phi_j$ lies on the right for intersection composition or $\phi_j$ is included in $\phi_i$ for inclusion composition.  The ambiguities of all  possible compositions in $\dRB$  are listed  as below: for arbitrary  $u, v, w \in \frakS(Z)$ and $q \in \frakS^{\star}(Z)$,
\begin{itemize}
	\item [$1 \wedge 1$] \quad $\underline{P(u)P(v)P(w)}$, \quad $P\left(\left.q\right|_{P(u)P(v)} \right)P\left(w\right)$, \quad  $P\left(u\right) P\left(\left.q\right|_{P(v)P(w)}\right)$

	\item [$1 \wedge 2$] \quad $P\left(\left.q\right|_{D(u)D(v)} \right)P\left(w\right)$, \quad $P\left(u\right) P\left(\left.q\right|_{D(u)D(v)}\right)$,

	\item [$1 \wedge 3$] \quad $P\left(\left.q\right|_{D(P(u))} \right)P\left(v\right)$, \quad $P\left(u\right) P\left(\left.q\right|_{D(P(v))}\right)$,

		\item [$1 \wedge 4$] \quad $\underline{P(u)P(v)D(w)}$,\quad $P\left(\left.q\right|_{P(u)D(v)} \right)P\left(w\right)$, \quad $P\left(u\right) P\left(\left.q\right|_{P(v)D(w)}\right)$,

		\item [$1 \wedge 5$] \quad $P\left(\left.q\right|_{D(u)P(v)} \right)P\left(w\right)$, \quad $P\left(u\right) P\left(\left.q\right|_{D(v)P(w)}\right)$,

\item [$2 \wedge 1$] \quad $D\left(\left.q\right|_{P(u)(P(v)}\right)D\left(w\right)$,\quad $D\left(u\right)D\left(\left.q\right|_{P(v)P(w)}\right)$,
	
	\item [$2 \wedge 2$] \quad$\underline{D(u)D(v)D(w)}$ , \quad $D\left(\left.q\right|_{D(u)D(v)}\right)D\left(w\right)$,\quad $D\left(u\right)D\left(\left.q\right|_{D(v)D(w)}\right)$,

		\item [$2 \wedge 3$] \quad $\underline{D(P(u))D(v)}$, \quad $\underline{D(u)D(P(v))}$,\quad $D\left(\left.q\right|_{D(P(u))}\right)D\left(v\right)$,\quad $D\left(u\right)D\left(\left.q\right|_{D(P(v))}\right)$,

		\item [$2 \wedge 4$] \quad  $D\left(\left.q\right|_{P(u)D(v)}\right)D\left(w\right)$,\quad $D\left(u\right)D\left(\left.q\right|_{P(v)D(w)}\right)$,

		\item [$2 \wedge 5$] \quad $\underline{D(u)D(v)P(w)}$,\quad  $D\left(\left.q\right|_{D(u)P(v)}\right)D\left(w\right)$,\quad $D\left(u\right)D\left(\left.q\right|_{D(v)P(w)}\right)$,

		\item [$3 \wedge 1$] \quad $D\left(P\left(\left.q\right|_{P(u)P(v)}\right)\right)$,

		\item [$3 \wedge 2$] \quad $D\left(P\left(\left.q\right|_{D(u)D(v)}\right)\right)$,

		\item [$3 \wedge 3$] \quad  $D\left(P\left(\left.q\right|_{D(P(u))}\right)\right)$,

		\item [$3 \wedge 4$] \quad $D\left(P\left(\left.q\right|_{P(u)D(v)}\right)\right)$,

		\item [$3 \wedge 5$] \quad $D\left(P\left(\left.q\right|_{D(u)P(v)}\right)\right)$,

		\item [$4 \wedge 1$] \quad $P\left(\left.q\right|_{P(u)P(v)}\right)D\left(w\right)$,\quad $P\left(u\right)D\left(\left.q\right|_{P(v)P(w)}\right)$,

\item [$4 \wedge 2$] \quad $\underline{P(u)D(v)D(w)}$, \quad $P\left(\left.q\right|_{D(u)D(v)}\right)D\left(w\right)$,\quad $P\left(u\right)D\left(\left.q\right|_{D(v)D(w)}\right)$,

    \item [$4 \wedge 3$] \quad $\underline{P(u)D(P(v))}$, \quad $P\left(\left.q\right|_{D(P(u))}\right)D\left(v\right)$,\quad $P\left(u\right)D\left(\left.q\right|_{D(P(v))}\right)$,

        \item [$4 \wedge 4$] \quad $P\left(\left.q\right|_{P(u)D(v)}\right)D\left(w\right)$,\quad $P\left(u\right)D\left(\left.q\right|_{P(v)D(w)}\right)$,

	\item [$4 \wedge 5$] \quad $\underline{P(u)D(v)P(w)}$, \quad $P\left(\left.q\right|_{D(u)P(v)}\right)D\left(w\right)$,\quad $P\left(u\right)D\left(\left.q\right|_{D(v)P(w)}\right)$,

		\item [$5 \wedge 1$] \quad $\underline{D(u)P(v)P(w)}$, \quad $D\left(\left.q\right|_{P(u)P(v)}\right)P\left(w\right)$,\quad $D\left(u\right)P\left(\left.q\right|_{P(v)P(w)}\right)$,

		\item [$5 \wedge 2$] \quad $D\left(\left.q\right|_{D(u)D(v)}\right)P\left(w\right)$,\quad $D\left(u\right)P\left(\left.q\right|_{D(v)D(w)}\right)$,

		\item [$5 \wedge 3$] \quad $\underline{D(P(u))P(v)}$,\quad $D\left(\left.q\right|_{D(P(u))}\right)P\left(v\right)$,\quad $D\left(u\right)P\left(\left.q\right|_{D(P(v))}\right)$,

\item [$5 \wedge 4$] \quad $\underline{D(u)P(v)D(w)}$,\quad $D\left(\left.q\right|_{P(u)D(v)}\right)P\left(w\right)$,\quad $D\left(u\right)P\left(\left.q\right|_{P(v)D(w)}\right)$,

\item [$5 \wedge 5$] \quad $D\left(\left.q\right|_{D(u)P(v)}\right)P\left(w\right)$,\quad $D\left(u\right)P\left(\left.q\right|_{D(v)P(w)}\right)$.
 \end{itemize}

Notice that all compositions above but the    underlined ones  can be dealt with by Lemma~\ref{including}.
There remains to consider the  underlined compositions. We only give the complete proof for the case $5 \wedge 1$, the other cases being similar.
For the case $5 \wedge 1$,  write
 $f = \phi_5(u,v)$, $g= \phi_1(v,w)$ and $p= D(u)P(v)P(w)$ . So we have
$$
\begin{aligned}
(f,g)_p^{P(w),D(u)} &=-D(uP(v))P(w)+uvP(w)+  \lambda D(u)vP(w)\\
& \quad  +D(u)P(vP(w))+D(u)P(P(v)w)+\lambda D(u)P(vw) \\
& \equiv -D(uP(v)P(w))+uP(v)w +\lambda D(uP(v))w+uvP(w)+  \lambda D(u)vP(w) \\
& \quad +D(uP(P(v)w))-uP(v)w- \lambda D(u)P(v)w\\
& \quad  +D(uP(vP(w)))-uvP(w)-\lambda D(u)vP(w)\\
& \quad  +\lambda D(uP(vw))-\lambda uvw-\lambda^2 D(u)vw\\
& \equiv -D(uP(v)P(w)) +D(uP(vP(w)))+D(uP(P(v)w))+\lambda D(uP(vw))\\
& \quad - \lambda D(u)P(v)w+\lambda D(uP(v))w   -\lambda uvw-\lambda^2 D(u)vw \\
& = -D(u\phi_1(v,w))-\lambda \phi_5(u,v)w\\
& \equiv 0 \bmod \left(S_{\dRB}(Z), p\right).
\end{aligned}
$$
We are done.
\end{proof}

 \begin{thm}
 	Let $Z$ be a set, $A=\bk \cals(Z)\slash I_A$ a nonunital $\bk$-algebra.  Then we have:
	$$\calf^{\dRB\zhx\Alg}_{\Alg}(A)=\bk\frakS(Z)\slash\left\langle S_{\dRB}(Z)\cup I_A\right\rangle_\mtOpAlg.$$
	Moreover, assume $I_A$  has a GS  basis $G$ with respect to the degree-lexicographical order $\leq_{\rm {dlex}}$. Then $S_{\dRB}(Z)\cup G$ is an operated  GS  basis of $\left\langle S_{\dRB}(Z)\cup I_A\right\rangle_\mtOpAlg$ in $\bk\frakS(Z) $  with respect to $\leq_{\operatorname{PD}}$.

 \end{thm}

 \begin{proof}
 Since the leading monomial in  $\dRB$ has no subword in $\cals(X)\backslash X$, the  result follows immediately from   Theorem~\ref{S2} and Theorem~\ref{Thm: GS basis for free nonunital Phi algebra over nonunital alg}.
 \end{proof}

As a consequence, we obtain a linear basis.
\begin{thm}
	Let $Z$ be a set, $A=\bk \cals(Z)\slash I_A$ a nonunital $\bk$-algebra with a GS  basis $G$ with respect to $\leq_{\rm{dlex}}$.   Then  the set
 $\Irr(S_{\dRB}(Z)\cup G)$ which is by definition the complement of
 $$\left\lbrace q|_{\bar{s}},q|_{P(u)P(v)},q|_{D(u)D(v)},q|_{D(P(u))},q|_{P(u)D(v)}, q|_{D(u)P(v)}, s\in G,q\in\frakSstar(Z),u,v\in\frakS(Z)\right\rbrace$$ in $\frakS(Z)$   is a linear basis of the free nonunital $\lambda$-differential Rota-Baxter algebra   $\calf^{\dRB\zhx\Alg}_{\Alg}(A)$ over $A$.

\end{thm}

\begin{proof}
It can be induced directly by   Theorem~\ref{Thm: nonunital CD}.
\end{proof}

\begin{remark}  Since the monomial order  used in \cite{BokutChenQiu} does not satisfy the conditions of  Theorem~\ref{Thm: GS basis for free nonunital Phi algebra over nonunital alg}, we have to make use of a new monomial order while treating free  differential Rota-Baxter algebras over an algebra.  In fact, since the  leading monomials are different,   even for free differential Rota-Baxter algebras over a field, our monomial order will provide new operated GS basis and linear basis.

\end{remark}

\subsection{ Case of nonunital algebras with $\lambda=0$}\label{Subsection: case zero}\

Now we consider
nonunital free differential Rota-Baxter algebras on algebras with $\lambda=0$.
This case can be studied similarly to the case $\lambda\neq 0$, so we omit the details in this subsection.



 Denote $\phi_1(x,y)$   with $\lambda=0$ by $\phi_1^0(x,y)$ . Let
$$\phi_2^0(x,y)=D(x)y+xD(y)-D(xy).$$
We also write $\phi_3^0(x)=\phi_3(x)$ for convenience.

\begin{prop}For any $u,v \in \frakS(Z)$, the leading monomials of  $\phi_1^0(u, v)$,  $\phi_2^0(u, v)$ and $\phi_3^0(u)$ with respect to  $\leq_{\operatorname{PD}}$ are
  $ P(u)P(v),  D(u)v$ and $D(P(u))$ respectively.

\end{prop}

Let $$\OdRB' :=\left\lbrace ~ \phi_1^0(x,y) , \phi_2^0(x,y) , \phi_3^0(x) \right\rbrace. $$
By the following example, one can see that  $\OdRB'$  is not  $\Omega$-GS in $\bk\frakS(Z)$ with respect to $\leq_{\operatorname{PD}}$.

\begin{exam}
For $u, v\in \frakS(Z)$, let
 $$\begin{array}{rcl}  f&=&\phi_2^0(P(u),v)=D(P(u))v+P(u)D(v)-D(P(u)v),\\\
   g &=& \phi_3^0(u)=D(P(u))-u,\\
   q&=&\star v,\\
   p&=&D(P(u))v=\bar{f}=\left.q\right|_{\bar{g}}. \end{array}$$
Then $$
(f,g)_p^q=f-\left.q\right|_{g}\equiv P(u)D(v)-D(P(u)v)+uv.
$$
Let $$\phi_4^0(x,y)=P(x)D(y)-D(P(x)y)+xy.$$  It is clear  that the leading monomial  of $\phi_4^0(u,v)$ with $u,v \in \frakS(Z)$ is   $P(u)D(v)$ with respect to $\leq_{\operatorname{PD}} $ which cannot  be reduced further.
\end{exam}

Now denote  $\OdRB$ to be the set of the following OPIs:
 	\begin{itemize}
		\item [(1)] $\phi_1^0(x,y) =P(x)P(y) - P(xP(y)) - P(P(x)y)$,

		\item [(2)] $ \phi_2^0(x,y)=D(x)y + xD(y) - D(xy) $,

		\item [(3)] $\phi_3^0(x)=D(P(x)) -x$,

		\item [(4)] $\phi_4^0(x,y)=P(x)D(y)-D(P(x)y)+xy$.
 	\end{itemize}
 It is obvious that
$\left\langle S_{\OdRB'}(Z)\right\rangle_\mtOpAlg=\left\langle S_{\OdRB}(Z)\right\rangle_\mtOpAlg$ for arbitrary set $Z$.

Similar to the case $\lambda\neq0$, it can be
proved that $\OdRB$ is $\Omega$-GS with respect to $\leq_{\operatorname{PD}}$.

\begin{remark}
Note that  $\phi_4^0(x,y)$ is just $\phi_4(x,y)$ with $\lambda=0$, and for $u,v \in \frakS(Z)$
$$
\phi_2^0(u,P(v)) =D(u)P(v) + uD(P(v)) - D(uP(v)) \equiv D(u)P(v) + uv - D(uP(v)),
$$
which is exactly  $\phi_5(u,v)$ with $\lambda=0$. So $\phi_5(x,y)$  ($\lambda=0$) does not appear in $\OdRB$.
\end{remark}

\begin{thm}\label{S1}
$\OdRB$ is  $\Omega$-GS  in $\bk\frakS(Z)$ with respect to  $\leq_{\operatorname{PD}} $.
\end{thm}

\begin{proof} As in the proof of  Theorem~\ref{S2},
we write  $i \wedge j$ the composition of OPIs of  $\phi_i$ and $\phi_j$.  There are two kinds of     ambiguities of all  possible compositions in $\OdRB$. Since $\phi_1^0(x,y)$,
$\phi_3^0(x)$, and
$\phi_4^0(x,y)$ have the same leading monomials as in the case $\lambda\neq 0$, the corresponding  ambiguities $i\wedge j$ with $ i, j\in \{1, 3, 4\}$ are the same in the proof of  Theorem~\ref{S2}.
%
%
%
%
%
%
%
%
%
%
%
%
%
%
%
%
%
%
%
%
%
            Since $\phi^0_2(x,y)$ has a different leading monomial, the ambiguities of the case $i\wedge j$ with $i=2$ or $j=2$ are the following:         for arbitrary $u,v,w \in \frakS(Z), q \in \frakS^{\star}(Z) $ and $s \in \frakS(Z) \text{ or } \emptyset$,
 \begin{itemize}

	\item [$1 \wedge 2$] \quad $P\left(\left.q\right|_{D(u)v}\right) P\left(w\right)$, \quad $P\left(u\right) P\left(\left.q\right|_{D(v)w}\right)$;

\item [$2 \wedge 1$] \quad $D\left(\left.q\right|_{P(u)P(v)}\right)w$,\quad $D\left(u\right)\left.q\right|_{P(u)P(v)}$;
	
	\item [$2 \wedge 2$] \quad $\underline{D(u)sD(v)w}$, \quad $D\left(\left.q\right|_{D(u)v}\right)w$,\quad $D\left(u\right)\left.q\right|_{D(v)w}$;

		\item [$2 \wedge 3$] \quad  $\underline{D\left(P(u)\right)v}$, \quad$D\left(\left.q\right|_{D(P(u))}\right)v$,\quad $D\left(u\right)\left.q\right|_{D(P(v))}$;

		\item [$2 \wedge 4$] \quad $\underline{D(u)sP(v)D(w)}$, \quad $D\left(\left.q\right|_{P(u)D(v)}\right)w$,\quad  $D\left(u\right)\left.q\right|_{P(v)D(w)}$;

		\item [$3 \wedge 2$] \quad $D\left(P\left(\left.q\right|_{D(u)v}\right)\right)$;

		\item [$4 \wedge 2$] \quad $\underline{P(u)D(v)w}$, \quad $P\left(\left.q\right|_{D(u)v}\right)D\left(w\right)$,\quad $P\left(u\right)D\left(\left.q\right|_{D(v)w}\right)$.
 \end{itemize}

Almost all the cases can be treated  similarly as in the proof of Theorem~\ref{S2}, except a slight difference in  the case   $2\wedge 2$. In fact,
   let
 $f = \phi_2^0(u,sD(v))$, $g= \phi_2^0(v,w)$ and $p=D(u)sD(v)w$. So we have
$$
\begin{aligned}
(f,g)_p^{D(u)s,D(w)} &=uD(sD(v))w-D(usD(v))w-D(u)svD(w)+D(u)sD(vw)\\
& \equiv -usD(v)D(w)+uD(sD(v)w)+usD(v)D(w)-D(usD(v)w)\\
& \quad +uD(svD(w))-D(usvD(w))-uD(sD(vw))+D(usD(vw))\\
& = uD(s\phi_2^0(v,w))-D(us\phi_2^0(v,w))\\
& \equiv 0 \bmod \left(S_{\OdRB}(Z), p\right).
\end{aligned}
$$
We are done.
\end{proof}

 \begin{thm}
 	Let $Z$ be a set and $A=\bk \cals(Z)\slash I_A$ a nonunital $\bk$-algebra.  Then we have:
	$$\calf^{\OdRB\zhx\Alg}_{\Alg}(A)=\bk\frakS(Z)\slash\left\langle S_{\OdRB}(Z)\cup I_A\right\rangle_\mtOpAlg.$$
	Moreover, assume $I_A$  has a GS  basis $G$ with respect to the degree-lexicographical order $\leq_{\rm {dlex}}$. Then $S_{\OdRB}(Z)\cup G$ is an operated  GS  basis of $\left\langle S_{\OdRB}(Z)\cup I_A\right\rangle_\mtOpAlg$ in $\bk\frakS(Z) $  with respect to $\leq_{\operatorname{PD}}$.
 \end{thm}

As a consequence, we obtain a linear basis.
\begin{thm}
	Let $Z$ be a set and $A=\bk \cals(Z)\slash I_A$ a nonunital $\bk$-algebra with a GS  basis $G$ with respect to $\leq_{\rm{dlex}}$.   Then  the set
 $\Irr(S_{\OdRB}(Z)\cup G)$ which is by definition the complement of
 $$\left\lbrace q|_{\bar{s}},q|_{P(u)P(v)},q|_{D(u)v},q|_{D(P(u))},q|_{P(u)D(v)}, s\in G,q\in\frakSstar(Z),u,v\in\frakS(Z)\right\rbrace$$ in $\frakS(Z)$   is a linear basis of the free nonunital $0$-differential Rota-Baxter algebra   $\calf^{\OdRB\zhx\Alg}_{\Alg}(A)$ over $A$.

\end{thm}

\subsection{ Case of unital   algebras}\label{Subsection: case unital   algebras}\

Now we consider  unital differential Rota-Baxter algebras. Since the proofs are similar to those in the previous subsections, we omit most of them.
The study still divided into  cases of $\lambda \neq 0$ and $\lambda=0$.

\medskip

When $\lambda \neq 0$, since unital differential Rota-Baxter algebras have the condition    $D(1)=0$, put $\udRB$ to be the union of $\dRB$  with  $\{D(1)\}$, but by abuse of notation,  in $\dRB$,      $x,y$ take their values in $  \frakM(Z)$ instead of $\frakS(Z)$.

\begin{remark}
We have:
$$
\left\{\begin{array}{lll}
\phi_2(u,v)\equiv 0  & \text{ when } u=1 \text{ or } v=1; \\
\phi_4(u,v)\equiv -D(P(u))+u=-\phi_3(u)  & \text{ when }  v=1; \\
\phi_5(u,v)\equiv -D(P(v))+v=-\phi_3(v)  & \text{ when }  u=1.
\end{array}\right.
$$
So adding of the unity $1$ will not produce new OPIs.
 Moreover,
it is clear that except the above cases, the leading monomials of OPIs in  $\dRB$ are  the same  with respect to $\leq_{\operatorname{PD}}$ and  $\leq_{\operatorname{uPD}}$ by Proposition~\ref{uDl}.
\end{remark}

With similar proofs as in Subsection~\ref{Subsection: Case of lambda non zero}, we can prove the following results.

\begin{thm}\label{uS1}
$\udRB$ is $\Omega$-GS in $\bk\frakM(Z)$ with respect to $\leq_{\operatorname{uPD}} $.
\end{thm}

 \begin{thm}
 	Let $Z$ be a set and $A=\bk \calm(Z)\slash I_A$ a unital $\bk$-algebra.  Then we have:
	$$\calf^{\udRB\zhx\uAlg}_{\uAlg}(A)=\bk\frakM(Z)\slash\left\langle S_{\udRB}(Z)\cup I_A\right\rangle_\mtuOpAlg.$$
	Moreover, assume $I_A$  has a GS  basis $G$ with respect to the degree-lexicographical order $\leq_{\rm {dlex}}$. Then $S_{\udRB}(Z)\cup G$ is an operated  GS  basis of $\left\langle S_{\udRB}(Z)\cup I_A\right\rangle_\mtuOpAlg$ in $\bk\frakM(Z) $  with respect to $\leq_{\operatorname{uPD}}$.

 \end{thm}

\begin{thm}
	Let $Z$ be a set and $A=\bk \calm(Z)\slash I_A$ a unital $\bk$-algebra with a GS  basis $G$ with respect to $\leq_{\rm{dlex}}$.   Then  the set
 $\Irr(S_{\udRB}(Z)\cup G)$ which is by definition the complement of
 $$\left\lbrace q|_{\bar{s}},q|_{P(u)P(v)},q|_{D(u)D(v)},q|_{D(P(u))},q|_{P(u)D(v)}, q|_{D(u)P(v)}, q|_{D(1)}, s\in G,q\in\frakMstar(Z),u,v\in\frakM(Z)\right\rbrace$$ in $\frakM(Z)$   is a linear basis of the free unital $\lambda$-differential Rota-Baxter algebra   $\calf^{\udRB\zhx\uAlg}_{\uAlg}(A)$ over $A$.

\end{thm}

\medskip

When $\lambda= 0$,  denote $\OudRB:=\OdRB$ (again   by abuse of notation,   $\OdRB$ is understood that     $u, v$ take their values in $  \frakM(X)$ instead of $\frakS(X)$).

 \begin{remark}
In $\OudRB$, we have
$$
\phi_2^0(1,1)=D(1)+D(1)-D(1)=D(1),
$$
so it is not necessary to add $D(1)$ into $\OudRB$.

 Note that $\phi_4^0(u,1) \equiv -D(P(v))+v=-\phi_3^0(v)$, so adding the unity $1$ will not induce any new OPI.
\end{remark}

By using similar proofs in Subsection~\ref{Subsection: case zero}, one can show the following results.

\begin{thm}\label{uS2}
$\OudRB$ is  $\Omega$-GS  in $\bk\frakM(Z)$ with $\leq_{\operatorname{uPD}} $.
\end{thm}

 \begin{thm}
 	Let $Z$ be a set and $A=\bk \calm(Z)\slash I_A$ a unital $\bk$-algebra.  Then we have:
	$$\calf^{\OudRB\zhx\uAlg}_{\uAlg}(A)=\bk\frakM(Z)\slash\left\langle S_{\OudRB}(Z)\cup I_A\right\rangle_\mtuOpAlg.$$
	Moreover, assume $I_A$  has a GS  basis $G$ with respect to the degree-lexicographical order $\leq_{\rm {dlex}}$. Then $S_{\OudRB}(Z)\cup G$ is an operated  GS  basis of $\left\langle S_{\OudRB}(Z)\cup I_A\right\rangle_\mtuOpAlg$ in $\bk\frakM(Z) $  with respect to $\leq_{\operatorname{uPD}}$.

 \end{thm}

\begin{thm}\label{different RB}
	Let $Z$ be a set and $A=\bk \calm(Z)\slash I_A$ a unital $\bk$-algebra with a GS  basis $G$ with respect to $\leq_{\rm{dlex}}$.   Then  the set
 $\Irr(S_{\OudRB}(Z)\cup G)$ which is by definition the complement of
 $$\left\lbrace q|_{\bar{s}},q|_{P(u)P(v)},q|_{D(u)v},q|_{D(P(u))},q|_{P(u)D(v)},  s\in G,q\in\frakMstar(Z),u,v\in\frakM(Z)\right\rbrace$$ in $\frakM(Z)$   is a linear basis of the free unital $0$-differential Rota-Baxter algebra   $\calf^{\OudRB\zhx\uAlg}_{\uAlg}(A)$ over $A$.

\end{thm}

So far, we have completed the study of differential Rota-Baxter algebras.

\section{Free integro-differential algebras over algebras }\label{Section: integro-differential algebras}

In this section, we carry the study of GS bases of free integro-differential algebras over algebras.
It reveals that integro-differential  algebras can be investigated by using a   method similar to differential Rota-Baxter algebras, but the details are  more difficult.


We first recall the definition of integro-differential algebras.

\begin{defn} Let $\lambda\in \bfk$.
 An integro-differential $\mathbf{k}$-algebra of weight $\lambda$ (also called a $\lambda$-integro-differential $\mathbf{k}$-algebra) is a differential $\mathbf{k}$-algebra $(R, d)$ of weight $\lambda$ with a linear operator $P: R \rightarrow$ $R$ which satisfies (c) in Definition~\ref{differentila rota-baxter algebras}:
 $$
D \circ P=\text { id },
$$
 and such that
$$
\begin{array}{ll}
P(D(u) P(v))=u P(v)-P(u v)-\lambda P(D(u) v) & \text { for all } u, v \in R, \\
P(P(u) D(v))=P(u) v-P(u v)-\lambda P(u D(v)) & \text { for all } u, v \in R.
\end{array}
$$
\end{defn}

Recall that
\begin{itemize}
     \item[(1)]  $\phi_1(x,y) = P(x) P(y)-P(x P(y))-P(P(x) y)-\lambda P(x y)$,
   \item[(2)]  $\phi_2(x,y)= D(x)D(y) + \lambda^{-1}D(x)y + \lambda^{-1}xD(y) - \lambda^{-1}D(xy) $,
    \item[(3)]  $\phi_3(x) = D(P(x))-x$,
\item [(4)] $\phi_4(x,y)=P(x)D(y)-D(P(x)y)+xy+\lambda xD(y)$,

		\item [(5)] $\phi_5(x,y)=D(x)P(y)-D(xP(y))+xy+\lambda D(x)y$,

\end{itemize}
and
denote
\begin{itemize}
     \item[(6)] $\phi_6(x,y) = P(D(x) P(y)) - x P(y)+P(x y)+\lambda P(D(x) y)$,
	  \item[(7)] $\phi_7(x,y) = P(P(x) D(y))-P(x) y+P(x y)+\lambda P(x D(y))$.
\end{itemize}

Notice that for  $u, v\in\frakS(Z)$,
since $P(D(u) P(v))$ (resp. $P(P(u) D(v))$) has the largest $P$-degree in $\phi_6(u,v)$ (resp. $\phi_7(u,v)$),  the leading monomial of $\phi_6(u,v)$ (resp. $\phi_7(u,v)$) with respect to $\leq_{\operatorname{PD}}$ is $P(D(u) P(v))$ (resp. $P(P(u) D(v))$).


\subsection{ Case of nonunital algebras with $\lambda\neq 0$}\

Assume in this subsection that $\lambda\neq 0$.
We first consider   nonunital free integro-differential $\bk$-algebras over  algebras.

According to the definition of integro-differential  algebras, define
$$
\inte' :=\left\lbrace ~  \phi_2(x, y) , \phi_3(x),  \phi_6(x, y) , \phi_7(x, y)~ \right\rbrace.
$$
By Example~\ref{phi_4}, Example~\ref{phi_5}, Example~\ref{phi8} and Example~\ref{phi9},   $\inte'$ is not    $\Omega$-GS in $\bk\frakS(X)$ with respect to $\leq_{\operatorname{PD}}$.

\begin{exam}\label{phi8}  For $u, v\in \frakS(Z)$, let
 $$\begin{array}{rcl}  f&=&\phi_7(u,v)=P(P(u) D(v))-P(u) v+P(u v)+\lambda P(u D(v)),\\\
   g &=& \phi_4(u,v)=P(u)D(v)-D(P(u)v)+uv+\lambda uD(v),\\
   q&=&P(\star),\\
   p&=&P(P(u) D(v))=\bar{f}=\left.q\right|_{\bar{g}}. \end{array}$$
Then $$
(f,g)_p^q=f-\left.q\right|_{g}\equiv -P(D(P(u)v))+P(u)v.
$$
Let$$
\phi_8(x,y)=P(D(P(x)y))-P(x)y.$$
 It is clear  that the leading monomial  of $\phi_8(u, v)$  is   $P(D(P(u)v))$ with respect to $\leq_{\operatorname{PD}}$ which cannot be reduced further.
\end{exam}

\begin{exam}\label{phi9}  For $u, v\in \frakS(Z)$, let
 $$\begin{array}{rcl}  f&=&\phi_6(u,v)=P(D(u) P(v)) - uP(v)+P(u v)+\lambda P(D(u) v),\\\
   g &=& \phi_5(u,v)=D(u)P(v)-D(uP(v))+uv+\lambda D(u)v,\\
   q&=&P(\star),\\
   p&=&P(D(u)P(v))=\bar{f}=\left.q\right|_{\bar{g}}. \end{array}$$
Then $$
(f,g)_p^q=f-\left.q\right|_{g}\equiv -P(D(uP(v)))+uP(v).
$$
Let$$
\phi_9(x,y)=P(D(xP(y)))-xP(y).$$
 It is clear  that the leading monomial  of $\phi_9(u, v)$ is   $P(D(uP(v)))$ with respect to $\leq_{\operatorname{PD}}$ which cannot be reduced further.
\end{exam}

\begin{remark}\label{phi1}
	Note that the OPI $\phi_1(x,y)$ can be induced by $\phi_3(x,y)$ and $\phi_6(x,y)$. So an integro-differential algebra can be seen as a differential Rota-Baxter algebra.
	Explicitly, for $u, v\in \frakS(Z)$, let
	$$\begin{array}{rcl}  f&=&\phi_6(P(u),v)=P(D(P(u)) P(v)) - P(u)P(v)+P(P(u) v)+\lambda P(D(P(u)) v),\\\
		g &=& \phi_3(u)=D(P(u))-u,\\
		q&=&P(\star P(v)),\\
		p&=&P(D(P(u)) P(v))=\bar{f}=\left.q\right|_{\bar{g}}. \end{array}$$
	Then $$
	(f,g)_p^q=f-\left.q\right|_{g}\equiv P(u)P(v) - P(uP(v)) - P(P(u)v) -  \lambda P(uv)=\phi_1(u,v).
	$$

\end{remark}

\bigskip

Now denote  $\inte$ to be the set of the following OPIs:
 	\begin{itemize}
		\item [(1)] $\phi_1(x, y) =P(x)P(y) - P(xP(y)) - P(P(x)y) -  \lambda P(xv)$,

		\item [(2)] $ \phi_2(x, y)=D(x)D(y) + \lambda^{-1}D(x)y + \lambda^{-1}xD(y) - \lambda^{-1}D(xy) $,

		\item [(3)] $\phi_3(x)=D(P(x)) - x$,

		\item [(4)] $\phi_4(x,y)=P(x)D(y)-D(P(x)y)+xy+\lambda xD(y)$,

		\item [(5)] $\phi_5(x,y)=D(x)P(y)-D(xP(y))+xy+\lambda D(x)y$,

		\item [(8)] $\phi_8(x,y)=P(D(P(x)y))-P(x)y$,

		\item [(9)] $\phi_9(x,y)=P(D(xP(y)))-xP(y)$.
 	\end{itemize}

Notice that  $\inte= \dRB\cup  \{\phi_8(x,y),  \phi_9(x,y)\}$.

\begin{prop}\label{s3s3'}
$\left\langle S_{\inte'}(Z)\right\rangle_\mtOpAlg=\left\langle S_{\inte}(Z)\right\rangle_\mtOpAlg$ for each set $Z$.
\end{prop}

\begin{proof}
We firstly prove $\left\langle S_{\inte}(Z)\right\rangle_\mtOpAlg  \subseteq  \left\langle S_{\inte'}(Z)\right\rangle_\mtOpAlg$,  which follows from
$$
\left\{\begin{array}{lll}
\phi_1(u, v) \in \left\langle \phi_3(u, v), \phi_6(u, v)\right\rangle_\mtOpAlg \   \mathrm{by\  Remark}~\ref{phi1},\\
\phi_4(u, v)  \in \left\langle \phi_2(u, v) , \phi_3(u)\right\rangle_\mtOpAlg\    \mathrm{by \  Example}~\ref{phi_4},\\
\phi_5(u, v)  \in \left\langle \phi_2(u, v) , \phi_3(u)\right\rangle_\mtOpAlg \  \mathrm{by \  Example}~\ref{phi_5},\\
\phi_8(u, v)  \in \left\langle \phi_4(u, v) , \phi_7(u, v) \right\rangle_\mtOpAlg\   \mathrm{by \  Example}~\ref{phi8},  \\
 \phi_9(u, v)  \in \left\langle \phi_5(u, v) , \phi_6(u, v) \right\rangle_\mtOpAlg\   \mathrm{by \  Example}~\ref{phi9},
\end{array}\right.
$$
where $ u,v\in \frakS(Z)$.

Next we show $  \left\langle S_{\inte'}(Z)\right\rangle_\mtOpAlg   \subseteq \left\langle S_{\inte}(Z)\right\rangle_\mtOpAlg $. Note that
$$
\begin{aligned}
P(\phi_4(u,v)) &=P(P(u)D(v))-P(D(P(u)v))+P(uv)+\lambda P(uD(v))\\
& =  P(P(u)D(v))-P(u)v+P(uv)+\lambda uD(v)-P(D(P(u)v))+P(u)v\\
& = \phi_7(u,v)-\phi_8(u,v),
\end{aligned}
$$
and
$$
\begin{aligned}
P(\phi_5(u,v)) &=P(D(u)P(v))-P(D(uP(v)))+P(uv)+\lambda P(D(u)v)\\
& =  P(D(u)P(v))-uP(v)+P(uv)+\lambda D(u)v-P(D(uP(v)))+uP(v)\\
& = \phi_6(u,v)-\phi_9(u,v).
\end{aligned}
$$
So we have
$$
\left\{\begin{array}{lll}
\phi_6(u,v) \in \left\langle \phi_5(u,v), \phi_9(u,v)\right\rangle_\mtOpAlg,\\
\phi_7(u,v) \in \left\langle \phi_4(u,v), \phi_8(u,v)\right\rangle_\mtOpAlg.
\end{array}\right.
$$
It proves $  \left\langle S_{\inte'}(Z)\right\rangle_\mtOpAlg   \subseteq \left\langle S_{\inte}(Z)\right\rangle_\mtOpAlg $.

We are done.
\end{proof}

Now we can prove  $\inte$ is  $\Omega$-GS.

\begin{thm}\label{S3,S4}
$\inte$ is   $\Omega$-GS  in $\bk\frakS(Z)$ with respect to $\leq_{\operatorname{PD}} $.
\end{thm}
\begin{proof}
  Since the     ambiguities  $i\wedge j$ with $i, j = 1, 2,3,4,5$ in $\inte$ are the same as in Theorem~\ref{S2}, we only need to consider the  ambiguities  involving  $\phi_8$ and $\phi_9$. The cases that cannot be dealt with directly by Lemma~\ref{including} are listed below: for arbitrary $u,v,w \in \frakS(Z), q\in\frakS^{\star}(Z)$ and $s \in \frakS(Z) \text{ or } \emptyset$,
 	\begin{itemize}
		\item [$1 \wedge 8$] \quad $P\left(D\left(P\left(u\right)v\right)\right)P\left(w\right)$, \quad $P\left(u\right)P\left(D\left(P\left(v\right)w\right)\right)$,

		\item [$3 \wedge 8$] \quad $D\left(P\left(D\left(P\left(u\right)v\right)\right)\right)$,

		\item [$4 \wedge 8$] \quad $P\big(D\left(P (u)v\right)\big)D\left(w\right)$,

		\item [$5 \wedge 8$] \quad $D\left(u\right)P\left(D\left(P\left(v\right)w\right)\right)$,

		\item [$1 \wedge 9$] \quad $P\big(D\left(uP(v)\right)\big)P\left(w\right)$, \quad $P\left(u\right)P\big(D\left(vP(w)\right)\big)$,

		\item [$3 \wedge 9$] \quad $D\Big(P\big(D\left(uP(v)\right)\big)\Big)$,

		\item [$4 \wedge 9$] \quad $P\big(D\left(uP(v)\right)\big)D\left(w\right)$,

		\item [$5 \wedge 9$] \quad $D\left(u\right)P\left(D\left(vP\left(w\right)\right)\right)$,

		\item [$8 \wedge 1$] \quad $P\big(D\left(P(u)P(v)s\right)\big)$,

		\item [$8 \wedge 4$] \quad $P\big(D\left(P(u)D(v)s\right)\big)$,

		\item [$9 \wedge 1$] \quad $P\big(D\left(sP(u)P(v)\right)\big)$,

		\item [$9 \wedge 5$] \quad $P\big(D\left(sD(u)P(v)\right)\big)$,

		\item [$8 \wedge 8$] \quad $P\left(D\Big(P\big(D\left(P(u)v\right)\big)w\Big)\right)$,

		\item [$8 \wedge 9$] \quad $P\left(D\Big(P\big(D\left(uP(v)\right)\big)w\Big)\right)$,

		\item [$9 \wedge 8$] \quad $P\left(D\Big(uP\big(D\left(P(v)w\right)\big)\Big)\right)$,

		\item [$9 \wedge 9$] \quad $P\left(D\Big(uP\big(D\left(vP(w)\big)\right)\Big)\right)$.
 	\end{itemize}
All these  compositions can be treated similarly as in the proof of  Theorem~\ref{S2}. We only give the complete proof for the case $8 \wedge 1$.
Take
 $f = \phi_8(u,P(v)s)$, $g= \phi_1(u,v)$, $p= P(D(P(u)P(v)s))$ and $q=P(D(\star s))$. Then we have
$$
\begin{aligned}
(f,g)_p^q &=-P(u)P(v)s+P(D(P(uP(v))s))+P(D(P(P(u)v)s))+\lambda P(D(P(uv)s))\\
& \equiv -P(uP(v))s-P(P(u)v)s-\lambda P(uv)s+P(uP(v))s+ P(P(u)v)s+\lambda P(uv)s\\
& \equiv 0 \bmod \left(S_{\inte}(Z), p\right).
\end{aligned}
$$
We are done.
\end{proof}

 \begin{thm}
 	Let $Z$ be a set and $A=\bk \cals(Z)\slash I_A$ a nonunital $\bk$-algebra.  Then we have:
	$$\calf^{\inte\zhx\Alg}_{\Alg}(A)=\bk\frakS(Z)\slash\left\langle S_{\inte}(Z)\cup I_A\right\rangle_\mtOpAlg.$$
	Moreover, assume $I_A$  has a GS  basis $G$ with respect to the degree-lexicographical order $\leq_{\rm {dlex}}$. Then $S_{\inte}(Z)\cup G$ is an operated  GS  basis of $\left\langle S_{\inte}(Z)\cup I_A\right\rangle_\mtOpAlg$ in $\bk\frakS(Z) $  with respect to $\leq_{\operatorname{PD}}$.

 \end{thm}

 \begin{proof}
 Since the leading monomial in  $\inte$ has no subword in $\cals(X)\backslash X$, the  result follows immediately from   Theorem~\ref{S3,S4} and Theorem~\ref{Thm: GS basis for free nonunital Phi algebra over nonunital alg}.
 \end{proof}

As a consequence, we obtain a linear basis .
\begin{thm}
	Let $Z$ be a set and $A=\bk \cals(Z)\slash I_A$ a nonunital $\bk$-algebra with a GS  basis $G$ with respect to $\leq_{\rm{dlex}}$.  Then a linear basis of the free nonunital $\lambda$-integro-differential algebra   $\calf^{\dRB\zhx\Alg}_{\Alg}(A)$ over $A$ is given by  the set
 $\Irr(S_{\inte}(Z)\cup G)$, which is by definition the complement in $\frakS(Z)$ of the subset consisting of  $  q|_w$ where  $w$ runs through
 $$  \bar{s},P(u)P(v),D(u)D(v),D(P(u)),P(u)D(v), D(u)P(v),P(D(P(u)v)),P(D(uP(v)))$$
 for arbitrary $ s\in G,q\in\frakSstar(Z),u,v\in\frakS(Z).$

\end{thm}

\begin{proof}
It can be induced directly from   Theorem~\ref{Thm: nonunital CD}.
\end{proof}

\begin{remark}Since the  monomial order  $\leq_{\operatorname{PD}}$ is different from that used in  \cite{GGR15}, our operated GS basis and linear basis are different from theirs. The reason is that   the monomial order in \cite{GGR15} does not satisfy the condition of Theorem~\ref{Thm: GS basis for free nonunital Phi algebra over nonunital alg}, thus cannot enable  us to discuss free integro-differential algebras over algebras.

\end{remark}

\begin{remark}
Define a new OPI $\phi_{10}(x)=P(D(x))$,
and let $$\Phi_{\mathsf{IID}}=\{~\phi_1(x,y),\phi_2(x,y),\phi_3(x),\phi_{10}(x)~\}.$$
A $\Phi_{\mathsf{IID}}$-algebra is just a nonunital $\lambda$-integro-differential algebra with the operators $P$ and $D$ being the inverse operator of each other, so we call such an operated algebra an invertible integro-differential algebra. One can show that
$\Phi_{\mathsf{IID}}\cup\{\phi_4(x,y),\phi_5(x,y)\}$ is  $\Omega$-GS  in $\bk\frakS(Z)$ with respect to $\leq_{\operatorname{PD}} $.

\end{remark}

\subsection{ Case of nonunital algebras with $\lambda=0$}\

Now we consider
nonunital free integro-differential algebras on algebras with $\lambda=0$.
 This case can be studied similarly as the case $\lambda\neq 0$, so we omit the details in this subsection.

As in Subsection~\ref{Subsection: case zero}, for a OPI $\phi$,  we denote $\phi^0$ for    $\phi$ with $\lambda=0$ and also write $\phi^0=\phi$ when  $\lambda$ does not appear in $\phi$ for convenience.
Let
$$
\Ointe' :=\left\lbrace ~ \phi_2^0(x, y) , \phi_3^0(x) , \phi_6^0(x, y), \phi_7^0(x, y) \right\rbrace.
$$
Again, $\Ointe'$  is not  $\Omega$-GS in $\bk\frakS(Z)$ with respect to $\leq_{\operatorname{PD}}$.

\begin{remark}
By  Example~\ref{phi8}, we can   get $\phi_8^0(u,v)$ from $\phi_4^0(u,v)$  and $\phi_7^0(u,v)$.

One can not obtain  $\phi_9^0(u,v)$ from $S_{\Ointe'}(Z)$ as in Example~\ref{phi9}, since $\phi_5$ does not belong to $\Ointe'$.
However, we can still generate $\phi_9^0(u,v)$ as follows:  for $u, v\in \frakS(Z)$, let
 $$\begin{array}{rcl}  f&=&\phi_6^0(u,v)=P(D(u)P(v))-uP(v)+P(u v),\\\
   g &=& \phi_2^0(u,P(v))=D(u)P(v)+uD(P(v))-D(uP(v)),\\
   q&=&P(\star),\\
   w&=&P(D(u)P(v))=\bar{f}=\left.q\right|_{\bar{g}}. \end{array}$$
Then $$
(f,g)_w=f-\left.q\right|_{g}\equiv P(D(uP(v)))-uP(v)=\phi_9^0(u,v).
$$

\end{remark}

Now denote  $\Ointe$ to be the set of the following OPIs:
 	\begin{itemize}
		\item [(1)] $\phi_1^0(x,y) =P(x)P(y) - P(xP(y)) - P(P(x)y)$,

		\item [(2)] $ \phi_2^0(x,y)=D(x)y + xD(y) - D(xy) $,

		\item [(3)] $\phi_3^0(x)=D(P(x)) - x$,

		\item [(4)] $\phi_4^0(x,y)=P(x)D(y)-D(P(x)y)+xy$,

		\item [(8)] $\phi_8^0(x,y)=P(D(P(x)y))-P(x)y$,

		\item [(9)] $\phi_9^0(x,y)=P(D(xP(y)))-xP(y)$.
 	\end{itemize}

 As in the previous subsection, one can prove the following results.
\begin{prop}
$\left\langle S_{\Ointe'}(Z)\right\rangle_\mtOpAlg=\left\langle S_{\Ointe}(Z)\right\rangle_\mtOpAlg$ for any set $Z$.
\end{prop}


\begin{thm}\label{S1ID}
$\Ointe$ is  $\Omega$-GS  in $\bk\frakS(Z)$ with respect to  $\leq_{\operatorname{PD}} $.
\end{thm}

 \begin{thm}
 	Let $Z$ be a set and $A=\bk \cals(Z)\slash I_A$ a nonunital $\bk$-algebra.  Then we have:
	$$\calf^{\Ointe\zhx\Alg}_{\Alg}(A)=\bk\frakS(Z)\slash\left\langle S_{\Ointe}(Z)\cup I_A\right\rangle_\mtOpAlg.$$
	Moreover, assume $I_A$  has a GS  basis $G$ with respect to the degree-lexicographical order $\leq_{\rm {dlex}}$. Then $S_{\Ointe}(Z)\cup G$ is an operated  GS  basis of $\left\langle S_{\Ointe}(Z)\cup I_A\right\rangle_\mtOpAlg$ in $\bk\frakS(Z) $  with respect to $\leq_{\operatorname{PD}}$.
 \end{thm}

\begin{thm}\label{integro}
	Let $Z$ be a set and $A=\bk \cals(Z)\slash I_A$ a nonunital $\bk$-algebra with a GS  basis $G$ with respect to $\leq_{\rm{dlex}}$.   Then  the set
 $\Irr(S_{\Ointe}(Z)\cup G)$ which is by definition the complement of
 $$\left\lbrace q|_{\bar{s}},q|_{P(u)P(v)},q|_{D(u)v},q|_{D(P(u))},q|_{P(u)D(v)},q|_{P(D(P(u)v))},q|_{P(D(uP(v)))}, s\in G,q\in\frakSstar(Z),u,v\in\frakS(Z)\right\rbrace$$ in $\frakS(Z)$   is a linear basis of the free nonunital $0$-integro-differential algebra $\calf^{\Ointe\zhx\Alg}_{\Alg}(A)$ over $A$.

\end{thm}

\subsection{ Case of unital   algebras}\

Now we consider  unital integro-differential algebras. Since the proofs are similar to those in the previous subsections, we omit most of them.
The study still is divided into  cases of $\lambda \neq 0$ and $\lambda=0$.

When $\lambda \neq 0$, since unital integro-differential algebras have the condition  $D(1)=0$, we put $\uinte:=\inte\cup \{D(1)\}$.

\begin{thm}\label{uS1ID}
$\uinte$ is  $\Omega$-GS  in $\bk\frakM(Z)$ with respect to $\leq_{\operatorname{uPD}} $.
\end{thm}

 \begin{thm}
 	Let $Z$ be a set and $A=\bk \calm(Z)\slash I_A$ a unital $\bk$-algebra.  Then we have:
	$$\calf^{\uinte\zhx\uAlg}_{\uAlg}(A)=\bk\frakM(Z)\slash\left\langle S_{\uinte}(Z)\cup I_A\right\rangle_\mtuOpAlg.$$
	Moreover, assume $I_A$  has a GS  basis $G$ with respect to the degree-lexicographical order $\leq_{\rm {dlex}}$. Then $S_{\uinte}(Z)\cup G$ is an operated  GS  basis of $\left\langle S_{\uinte}(Z)\cup I_A\right\rangle_\mtuOpAlg$ in $\bk\frakM(Z) $  with respect to $\leq_{\operatorname{uPD}}$.

 \end{thm}

\begin{thm}
	Let $Z$ be a set, $A=\bk \calm(Z)\slash I_A$ a unital $\bk$-algebra with a GS  basis $G$ with respect to $\leq_{\rm{dlex}}$.   Then a linear basis of the free unital $\lambda$-integro-differential algebra   $\calf^{\uinte\zhx\uAlg}_{\uAlg}(A)$ over $A$ is given by  the set
 $\Irr(S_{\uinte}(Z)\cup G)$,  which is by definition the complement  in $\frakM(Z)$ of  the subset consisting of  $  q|_w$ where  $w$ runs through
 $$  \bar{s},P(u)P(v),D(u)D(v),D(P(u)),P(u)D(v), D(u)P(v),P(D(P(u)v)),P(D(uP(v))), D(1)$$
 for arbitrary $ s\in G,q\in\frakMstar(Z),u,v\in\frakM(Z).$

\end{thm}

When $\lambda= 0$, denote $\Ouinte:=\Ointe$.

\begin{thm}\label{uS2ID}
$\Ouinte$ is  $\Omega$-GS  in $\bk\frakM(Z)$ with respect to $\leq_{\operatorname{uPD}} $.
\end{thm}

 \begin{thm}
 	Let $Z$ be a set and $A=\bk \calm(Z)\slash I_A$ a unital $\bk$-algebra.  Then we have:
	$$\calf^{\Ouinte\zhx\uAlg}_{\uAlg}(A)=\bk\frakM(Z)\slash\left\langle S_{\Ouinte}(Z)\cup I_A\right\rangle_\mtuOpAlg.$$
	Moreover, assume $I_A$  has a GS  basis $G$ with respect to the degree-lexicographical order $\leq_{\rm {dlex}}$. Then $S_{\Ouinte}(Z)\cup G$ is an operated  GS  basis of $\left\langle S_{\Ouinte}(Z)\cup I_A\right\rangle_\mtuOpAlg$ in $\bk\frakM(Z) $  with respect to $\leq_{\operatorname{uPD}}$.

 \end{thm}

\begin{thm}
	Let $Z$ be a set and $A=\bk \calm(Z)\slash I_A$ a unital $\bk$-algebra with a GS  basis $G$ with respect to $\leq_{\rm{dlex}}$.   Then  the set
 $\Irr(S_{\Ouinte}(Z)\cup G)$ which is by definition the complement of
 $$\left\lbrace q|_{\bar{s}},q|_{P(u)P(v)},q|_{D(u)v},q|_{D(P(u))},q|_{P(u)D(v)} ,q|_{P(D(P(u)v))},q|_{P(D(uP(v)))}, s\in G,q\in\frakMstar(Z),u,v\in\frakM(Z)\right\rbrace$$ in $\frakM(Z)$   is a linear basis of the free unital $0$-integro-differential algebra   $\calf^{\Ouinte\zhx\uAlg}_{\uAlg}(A)$ over $A$.
\end{thm}

\subsection{Differential Rota-Baxter algebras vs integro-differential algebras }\

Since integro-differential algebras have one more defining relation  than differential Rota-Baxter algebras,
by Proposition~\ref{free phi algebra}, the free integro-differential algebra over an algebra $A$ is a quotient of the free differential Rota-Baxter algebra over $A$ in general.  However, by using the descriptions of  $\dRB$  and $\inte$ and    Theorems~\ref{S2}    and \ref{S3,S4},  we could  also show that the former one is a differential Rota-Baxter algebra subalgebra of the later one.

\begin{thm}\label{thm: RB vs ID}
The free nonunital $\lambda$-integro-differential algebra $\calf^{\inte\zhx\Alg}_{\Alg}(A)$ over an algebra $A$ is differential Rota-Baxter subalgebra of the free nonunital $\lambda$-differential Rota-Baxter algebra $\calf^{\dRB\zhx\Alg}_{\Alg}(A)$ over $A$.
 \end{thm}

\begin{proof}
	We have the observation mentioned before
	$$\inte= \dRB\cup  \{\phi_8(x,y),  \phi_9(x,y)\}.$$
	That is to say, the operated Gr\"obner-Shirshov basis of the free nonunital $\lambda$-differential Rota-Baxter algebra $\calf^{\dRB\zhx\Alg}_{\Alg}(A)$ over an algebra $A$ is a subset of that of the free nonunital $\lambda$-integro-differential algebra $\calf^{\inte\zhx\Alg}_{\Alg}(A)$ over $A$. So by Diamond Lemma,  $\calf^{\inte\zhx\Alg}_{\Alg}(A)$ is a subspace of $\calf^{\dRB\zhx\Alg}_{\Alg}(A)$.  It is obvious that $\calf^{\inte\zhx\Alg}_{\Alg}(A)$ is also  differential Rota-Baxter subalgebra of $\calf^{\dRB\zhx\Alg}_{\Alg}(A)$.
\end{proof}

\begin{remark}
	  Gao and Guo \cite{GGR15}  also studied  GS bases of the free integro-differential algebras and free differential Rota-Baxter algebra both generated  by sets, and  they   deduced that  the free integro-differential algebra generated  by a set  is  a subalgebra of the free differential Rota-Baxter algebra generated by the same  set.
Theorem~\ref{thm: RB vs ID} proves an analogous    fact for  these free algebras   generated by algebras.
However, our   method is completely different from theirs.
\end{remark}

\begin{remark} By using the descriptions of  $\OdRB$ and $\Ointe$ (resp.  $\udRB$ and $\uinte$, $\OudRB$ and  $\Ouinte$)   and    Theorems~\ref{S1} and ~\ref{S1ID}  (resp. Theorems~\ref{uS1} and \ref{uS1ID}, Theorems~\ref{uS2} and  \ref{uS2ID}), we always have the same result
	  in both unital and nonunital cases with any  $\lambda$ (zero or nonzero).
 \end{remark}

\bigskip

 \textbf{Acknowledgements:}   The authors were  supported   by  NSFC (No. 11771085, 12071137) and by STCSM  (22DZ2229014).


\end{document}